\documentclass[12pt]{amsart}

\usepackage{amsmath}
\usepackage{amsfonts}
\usepackage{amssymb}
\usepackage{amsthm} 
\usepackage[english]{babel}
\usepackage[latin1]{inputenc}
\usepackage[T1]{fontenc}
\usepackage{graphicx}
\usepackage{enumitem}

\usepackage{geometry}
\usepackage{hyperref}

\DeclareMathOperator{\Ric}{Ric}

\DeclareMathOperator{\spt}{spt}

\newtheorem{theo}{Theorem}[]
\newtheorem{prop}[theo]{Proposition}
\newtheorem{lemme}[theo]{Lemma}
\newtheorem{definition}[theo]{Definition}

\newtheorem{remarque}[theo]{Remark}
\begin{document}

\title [Embeddedness of least area minimal hypersurfaces] {Embeddedness of least area minimal hypersurfaces}
\author{Antoine Song}

\begin{abstract}
In \cite{CC}, E. Calabi and J. Cao showed that a closed geodesic of least length in a two-sphere with nonnegative curvature is always simple. Using min-max theory, we prove that for some higher dimensions, this result holds without assumptions on the curvature. More precisely, in a closed $(n+1)$-manifold with $2 \leq n \leq 6$, a least area closed minimal hypersurface exists and any such hypersurface is embedded. \\
As an application, we give a short proof of the fact that if a closed three-manifold $M$ has scalar curvature at least $6$ and is not isometric to the round three-sphere, then $M$ contains an embedded closed minimal surface of area less than $4\pi$. This confirms a conjecture of F. C. Marques and A. Neves.
\end{abstract}

\maketitle
\begin{center}
\end{center}

\section*{}      

Let $(M^{n+1},g)$ be a closed $(n+1)$-manifold with $2\leq n \leq 6$. Consider a closed $n$-manifold ${\Gamma}$ and an immersion $\psi: \GammaÊ\to M$. If $\psi$ is minimal then $\psi(\Gamma) = \Sigma$ is called a minimal hypersurface and $\Sigma$ is said to be a minimal hypersurface of least area if for every minimal hypersurface $\Sigma'\subset M$, $$\mathcal{H}^n(\Sigma)\leq\mathcal{H}^n(\Sigma').$$ We will only consider closed minimal hypersurfaces. Minimal hypersurfaces which minimize the area among a family of \textit{embedded} minimal hypersurfaces have been previously examined in some situations related to min-max theory. In \cite{MaNe}, Marques and Neves initiated the characterisation of the area and the Morse index of surfaces produced by min-max theory in three-manifolds: among other things, they showed that in a closed oriented three-manifold with positive Ricci curvature, with Heegaard genus $h$ and which does not contain any embedded non-orientable surface, there is an index one minimal suface $\Sigma$ of genus $h$ produced by the smooth min-max theory (see \cite{C&DL}) such that
$$\mathcal{H}^2(\Sigma)=\inf_{S \in \mathcal{E}_h}\mathcal{H}^2(S),$$
where $\mathcal{E}_h$ denotes the collection of all connected {embedded} minimal surfaces of genus not larger than $h$. This result was later extended to higher dimensions by Zhou \cite{Zhou} in the following way. Let $M^{n+1}$ be a closed oriented manifold with $2\leq n \leq 6$. Let $\mathcal{O}$ be the collection of all {embedded} orientable closed minimal hypersurfaces in $M$ and $\mathcal{U}$ be the collection of the non orientable ones. Zhou studied the minimal hypersurfaces produced by Almgren-Pitts' theory applied to the fundamental class $[M]$, under the assumption of positive Ricci curvature. It turns out in particular that their area counted with multiplicity is 
$$\mathcal{A}_1(M) = \inf(\{\mathcal{H}^n(\Sigma); \Sigma \in \mathcal{O}\} \cup \{2\mathcal{H}^n(\Sigma); \Sigma \in \mathcal{U}\}).$$
Recently, Mazet and Rosenberg \cite{MR} proved that, without assumptions on the Ricci curvature, this area $\mathcal{A}_1(M)$ is realized by a hypersurface and classified all such hypersurfaces: either they are stable or they are of index one and produced by min-max constructions. The forementioned authors defined the quantity $\mathcal{A}_1(M)$ by separating orientable and non-orientable minimal hypersurfaces. This is due to technical reasons coming from their use of Almgren-Pitts' min-max theory. In this paper we do not need to make this distinction and we are interested in least area minimal hypersurfaces, where "least area" is taken in the geometric sense. The aim is to prove by min-max methods that in $M$, any least area minimal hypersurface is actually embedded. It extends a result of Calabi and Cao \cite{CC}.

\begin{theo} \label{principal}
Let $(M^{n+1},g)$ be a closed $(n+1)$-manifold with $2\leq n \leq 6$. Then there exists a least area closed minimal hypersurface $\Sigma \subset M$ and any such $\Sigma$ is embedded.
\end{theo}

We state in Section \ref{proof} a more detailed version, which precises that, similarly to \cite{MR}, the least area minimal hypersurfaces are either embedded stable hypersurfaces or given by Almgren-Pitts' min-max theory applied to the fundamental class $[M]$. Interestingly, no assumptions on the curvature are needed here whereas in dimension $n=1$ there are counterexamples if the curvature of the sphere is not nonnegative. This is related to the fact that min-max theory produces embedded hypersurfaces for $n=2,...,6$ but only immersed geodesics for $n=1$.

In \cite{CC}, Calabi and Cao proved that on a Riemannian sphere $\mathcal{S}$ with nonnegative curvature, any non-trivial closed geodesic of the shortest length is simple. Their proof is based on the fact that if $\gamma \subset \mathcal{S}$ is a closed geodesic which is not simple, then there is a nontrivial loop $\{\sigma_t\}_{t\in [0,1]}$ in the space of one-dimensional integral currents without boundary such that each $\sigma_t$ has mass less than the length of $\gamma$. Then by the min-max theory of Almgren and Pitts, there exists a geodesic strictly shorter than $\gamma$. Actually they noted that, in dimension one, the min-max principle can be shown simply without Geometric Measure Theory. Their construction of the "sweepout" family $\{\sigma_t\}$ relies in an essential way on the dimension of the problem: indeed for curves, there is an adhoc shortening procedure known as the Birkhoff curve shortening process. It turns out that there is already no equivalent process for surfaces. The proof of Theorem \ref{principal} will retain the main structure of the proof of Calabi and Cao, but we have to replace the Birkhoff process by a method developed by Marques and Neves in \cite{MaNe} when $n=2$ and extended to $2\leq n\leq 6$ by Zhou in \cite{Zhou}. The restriction on the dimension is a classical one due to the regularity results of Schoen and Simon \cite{SchoenSimon}. The structure of an immersed minimal hypersurface ($n\geq 2$) is less understood than the case of geodesics, and it will add some technical issues.

As a corollary of Theorem \ref{principal}, we can prove the following result conjectured by Marques and Neves in \cite{MaNe}: it asserts that a positive lower bound on the scalar curvature of a closed three-manifold $M$ gives a rigid upper bound on the area of the smallest embedded closed minimal surface in $M$.

\begin{theo} \label{conj}
Let $M^3$ be a closed three-manifold with scalar curvature $R$ at least $6$, not isometric to the round unit three-sphere $S^3$. Then there exists a closed embedded minimal surface $\Sigma$ of index zero or one such that
$$\mathcal{H}^2(\Sigma) < 4\pi.$$
\end{theo}

Marques and Neves already proved this theorem with the additional assumption that the Ricci curvature is positive. In their proof, they use the min-max theory of Simon-Smith and Hamilton's article on the Ricci flow for three-manifolds with positive Ricci curvature. Our proof is based on the theory of Almgren-Pitts and only uses the short-time existence theorem for the Ricci flow.

This paper is organized as follows: in Section \ref{prel} are reviewed the basic notions of Almgren-Pitts' theory as well as the continuous min-max theory, then some min-max constructions are presented and an outline of the principal proof sums up the strategy and the differences with \cite{CC}. The proof of Theorem \ref{principal} is given in Section \ref{proof} and we explain in Section \ref{conjec} how Theorem \ref{conj} is a consequence of the main theorem. Some technical results are proved in the Appendix. 

\subsection*{Acknowledgement} 
I am grateful to my advisor Fernando Cod{\'a} Marques for bringing a version of the main question to my attention. I would like to thank him for his constant support, for stimulating discussions and for guiding me through the recent literature. I also want to thank Harold Rosenberg for a meaningful discussion.

\section{Preliminaries} \label{prel}

In this section, we present a brief overview of the essential notions of Almgren and Pitts' theory, the continuous min-max theory and some useful results. Afterwards, a heuristic description of the main steps in the proof of Theorem \ref{principal} is given as well.

\subsection{} \textbf{Min-max theory in the setting of Almgren and Pitts}

For the convenience of the reader, we give a quick review of the basic definitions from Geometric Measure Theory and some notions of the Almgren and Pitts' theory used thereafter. For a complete presentation, one can refer to the book of Pitts \cite{P} or to Section 2 in \cite{MaNeinfinity}.

Let $M$ be a closed Riemannian $(n+1)$-manifold, assumed to be isometrically embedded in $\mathbb{R}^P$. We will always suppose $M$ connected. Because we will have to consider non-orientable submanifolds, we will use flat chains modulo $2$ (see \cite[4.2.26]{Federer}). We will work with the space $\mathbf{I}_k(M,\mathbb{Z}_2)$ of flat chains modulo $2$ in the congruence class of some $k$-dimensional integral currents with support contained in $M$, the subspace $\mathcal{Z}_k(M,\mathbb{Z}_2) \subset \mathbf{I}_k(M,\mathbb{Z}_2)$ whose elements have no boundary, and with the space $\mathcal{V}_k(M)$ of the closure, in the weak topology, of the set of $k$-dimensional rectifiable varifolds in $\mathbb{R}^P$ with support in $M$.

As in \cite{P}, we will suppress the superscript $2$ in the notation of \cite{Federer}. An integral current $T\in \mathbf{I}_k(M,\mathbb{Z}_2)$ determines an integral varifold $|T|$ and a Radon measure $||T||$ (\cite[Chapter 2, 2.1, (18) (e)]{P}). If $V \in \mathcal{V}_k(M)$, denote by $||V||$ the associated Radon measure on $M$. Given an open set $U\subset M$, if the associated rectifiable current is an integral current in $\mathbf{I}_{n+1}(M,\mathbb{Z}_2)$, it will be written as $[|U|]$. To a rectifiable subset $R$ of $M$ corresponds an integral varifold called $|R|$. The support of a current or a measure is denoted by $\spt$. The notation $\mathbf{M}$ stands for the mass of an element in $\mathbf{I}_k(M, \mathbb{Z}_2)$. On $\mathbf{I}_k(M,\mathbb{Z}_2)$ there is also the flat metric $\mathcal{F}(.,.)$ which induces the so-called flat topology. The space $\mathcal{V}_k(M)$ is endowed with the topology of the weak convergence of varifolds. The notations $\mathcal{Z}_k(M,\underline{\underline{\nu}},\mathbb{Z}_2)$ and $\mathbf{I}_k(M,\underline{\underline{\nu}},\mathbb{Z}_2)$ mean that the respective spaces of currents are considered with the topology induced by $\underline{\underline{\nu}}$, where $\underline{\underline{\nu}}$ is either $\mathbf{M}$ or $\mathcal{F}$.

We denote $[0,1]$ by $I$. For each $j\in \mathbb{N}$, $I(1,j)$ stands for the cell complex on $I$ whose $1$-cells and $0$-cells are, respectively,
$$[0,3^{-j}], [3^{-j},2.3^{-j}], ... , [1-3^{-j},1] \text{  and  } [0],[3^{-j}], ..., [1-3^{-j}],[1].$$
$I(1,j)_p$ denotes the set of all $p$-cells in $I(1,j)$.

In the theory of Almgren and Pitts, instead of considering continuous maps from $I$ to $\mathcal{Z}_n(M,\mathbb{Z}_2)$, one consider a sequence of mappings from $I(1,n_i)$ to $\mathcal{Z}_n(M,\mathbb{Z}_2)$, where $n_i\to \infty$ and the discrete slices corresponding to adjacent vertices in $I(1,n_j)$ are closer and closer. This leads to the following notions:

\begin{definition}
Whenever $\phi : I(1,j)_0 \to \mathcal{Z}_n(M,\underline{\underline{\nu}},\mathbb{Z}_2)$, we define the fineness of $\phi$ to be
$$\mathbf{f}_{\underline{\underline{\nu}}}(\phi) = \sup \bigg\{ \frac{\underline{\underline{\nu}}(\phi(x) - \phi(y))}{\mathbf{d}(x,y)} ; x,y \in I(1,j)_0,x\neq y\bigg\}$$
where $\mathbf{d}(x,y) = 3^j|x-y|.$
\end{definition}

For each $x\in I(1,j)_0$, define $\mathbf{n}(i,j)(x)$ to be the unique element of $I(1,j)_0$ such that $$\mathbf{d}(x,\mathbf{n}(i,j)(x)) = \inf \{\mathbf{d}(x,y) ; y\in I(1,j)_0\}.$$

\begin{definition}
\begin{enumerate}

\item Let $\delta>0$. We say that $\phi_1$ and $\phi_2$ are homotopic in $(\mathcal{Z}_n(M,\underline{\underline{\nu}},\mathbb{Z}_2),\{0\})$ with fineness $\delta$ if and only if there exist positive integers $k_1$, $k_2$, $k_3$ and a map
$$\psi : I(1,k_3)_0 \times I(1,k_3)_0 \to \mathcal{Z}_n(M,\underline{\underline{\nu}},\mathbb{Z}_2)$$
such that $\mathbf{f}_{\underline{\underline{\nu}}}(\psi)<\delta$ and whenever $j=1,2$ and $x\in I(1,k_3)_0$,
$$\phi_j : I(1,k_j) \to \mathcal{Z}_n(M,\underline{\underline{\nu}},\mathbb{Z}_2), \quad \phi_j ([0])=\phi_j([1]) =0,$$
$$\psi([j-1],x) = \phi_j(\mathbf{n}(k_3,k_j)(x)), \quad \psi(x,[0]) = \psi(x,[1]) = 0.$$

\item A $(1,\underline{\underline{\nu}})$-homotopy sequence of mappings into $(\mathcal{Z}_n(M,\underline{\underline{\nu}},\mathbb{Z}_2),\{0\})$ is a sequence $S=\{\phi_1,\phi_2,...\} $ for which there exist positive numbers $\delta_1,\delta_2...$ such that $\phi_i$ is homotopic to $\phi_{i+1}$ in $(\mathcal{Z}_n(M,\underline{\underline{\nu}},\mathbb{Z}_2),\{0\})$ with fineness $\delta_i$ for each positive integer $i$, $\lim_i\delta_i =0$ and 
$$\sup\{\mathbf{M}(\phi_i(x)) ; x \in \text{domain}(\phi_i), i>0\} <\infty.$$

\item If $S_1=\{\phi_i^1\}$ and $S_2=\{\phi_i^2\}$ are $(1,\underline{\underline{\nu}})$-homotopy sequences of mappings into $(\mathcal{Z}_n(M,\underline{\underline{\nu}},\mathbb{Z}_2),\{0\})$, then $S_1$ is homotopic with $S_2$ if and only if there is a sequence a positive real numbers $\delta_1,\delta_2,...$ such that $\phi_i^1$ is homotopic to $\phi_i^2$ in $(\mathcal{Z}_n(M,\mathbf{M}),\{0\})$  with fineness $\delta_i$ for $i>0$ and $\lim_i \delta_i =0$.

"To be homotopic with" is an equivalence relation on the set of $(1,\mathbf{M})$-homotopy sequences of mappings into $(\mathcal{Z}_n(M,\mathbb{Z}_2),\{0\})$. An equivalence class of such sequences is a $(1,\underline{\underline{\nu}})$-homotopy class of mappings into $(\mathcal{Z}_n(M,\underline{\underline{\nu}}),\{0\})$. The space of these equivalence classes is denoted by $$\pi_1^{\sharp}(\mathcal{Z}_n(M,\mathbb{Z}_2),\{0\}).$$

\end{enumerate}
\end{definition}

\begin{remarque} \label{isohomo}
By \cite[Theorem 4.6]{P}, 
$$\pi_1^{\sharp}(\mathcal{Z}_n(M,\mathbf{M},\mathbb{Z}_2),\{ 0\})\text{, }  \pi_1^{\sharp}(\mathcal{Z}_n(M,\mathcal{F},\mathbb{Z}_2),\{ 0\}) \text{ and } \pi_1(\mathcal{Z}_n(M,\mathcal{F},\mathbb{Z}_2),\{ 0\})$$ are all naturally isomorphic. 

\end{remarque}

Given $\Pi \in \pi_1^{\sharp}(\mathcal{Z}_n(M,\mathbf{M},\mathbb{Z}_2),\{0\})$, consider the function $\mathbf{L}: \Pi \to [0,\infty]$ defined such that if $S=\{\phi_i\}_{i \in \mathbb{M}} \in \Pi$ and $\phi_i : I(1,n_i) \to \mathcal{Z}_n(M,\mathbb{Z}_2)$, then 
$$\mathbf{L}(S) = \limsup_{i\to \infty} \max\{\mathbf{M}(\phi_i(x)) : x \in I(1,n_i)\}.$$
The width of $\Pi$ is then the following quantity:
$$\mathbf{L}(\Pi) = \inf\{\mathbf{L}(S) ; S \in \Pi\}.$$
A sequence $S=\{\phi_i :I(1,n_i) \to \mathcal{Z}_n(M,\mathbb{Z}_2)\}_{i} \in \Pi$ is said to be critical for $\Pi$ if $\mathbf{L}(S)=\mathbf{L}(\Pi)$. Define the critical set $\mathbf{C}(S) \subset \mathcal{V}_n(M)$ of $S \in \Pi$ as
\begin{align*}
\mathbf{C}(S) = \{V ;  \quad& \exists \{i_j\}_j, \exists \{x_j\}, i_j \to \infty, x_j \in I(1,n_{i_j}) ,\\  
& V = \lim\limits_{j\to \infty} |\phi_{i_j}(x_j)| \text{ and } ||V||(M)=\mathbf{L}(S)\}.
\end{align*}

We recall that one can define the $\mathbf{F}$-metric on $\mathcal{V}_n(M)$ (\cite[Chapter 2, 2.1, (19)]{P}). The $\mathbf{F}$-metric induces the weak topology on any set of varifolds whose mass is bounded by a certain constant. We now give Pitts' definition of almost minimizing varifolds (see \cite[Chapter 3]{P}). Let $U$ be an open set of $\mathbb{R}^L$. Denote by $\mathcal{Z}_n(M,M\backslash U,\mathbb{Z}_2)$ the set of currents $T\in \mathcal{Z}_n(M,\mathbb{Z}_2)$ such that $\spt (\partial T) \subset M\backslash U$. For each pair of positive numbers $\epsilon$ and $\delta$, we define
$$\mathfrak{a}_n(U,\epsilon,\delta,\mathbb{Z}_2)$$
to be the set of all currents $T\in \mathcal{Z}_n(M,M \backslash U,\mathbb{Z}_2)$ with the following property. If $T=T_0,$ $T_1$, ... , $T_m\in \mathcal{Z}_n(M,M \backslash U,\mathbb{Z}_2)$,
$$\spt (T-T_i) \subset U \quad \text{for } i=1, ... , m$$
$$\mathcal{F}(T_i,T_{i-1}) \leq \delta  \quad \text{for } i=1, ... , m$$
$$\mathbf{M}(T_i)\leq \mathbf{M}(T)+\delta  \quad \text{for } i=1, ... , m$$
then $\mathbf{M}(T)-\mathbf{M}(T_m)\leq \epsilon$.

\begin{definition}
We say that $V\in \mathcal{V}_n(M)$ is $\mathbb{Z}_2$ almost minimizing in $U$ if and only if for each $\epsilon>0$, there exist a $\delta>0$ and a current $T\in \mathfrak{a}_n(U,\epsilon,\delta,\mathbb{Z}_2)$ such that $\mathbf{F}(V,|T|) \leq \epsilon$.
\end{definition}

\begin{remarque} \label{modulo}
As we deal with flat chains modulo $2$, we will use the "modulo $2$" versions of results in \cite{P}, \cite{MaNeWillmore}, \cite{Zhou} (see \cite{MaNeinfinity} for a detailed explanation). We assume that $2\leq n\leq 6$. Note that Pitts proved \cite[Theorem 7.12]{P} for $2\leq n \leq 5$ but because of the curvature estimates in \cite[(7.2)]{SchoenSimon}, it still holds true for $n=6$.
\begin{enumerate}
\item If $V$ is a stationary $n$-varifold in $\mathcal{V}_n(M)$ and is $\mathbb{Z}_2$ almost minimizing in small annuli around each point, then $\spt V$ is a smooth embedded minimal hypersurface (\cite[Theorem 2.11]{MaNeinfinity}). This is the "modulo $2$" version of Theorem $7.12$ in \cite{P}.
\item The interpolation results \cite[Theorem 13.1]{MaNeWillmore}, \cite[Theorem 5.5]{Zhou} are true for $\mathbb{Z}_2$ (see \cite[Theorem 3.9]{MaNeinfinity} for a statement in the "modulo 2" setting).
\item \cite[Theorem 5.8]{Zhou} is also still true for $\mathbb{Z}_2$.
\end{enumerate}

\end{remarque}

Finally, we describe the isomorphism 
\begin{equation} \label{iso}
H_{n+1}(M,\mathbb{Z}_2)  \simeq   \pi_1(\mathcal{Z}_n(M,\mathcal{F},\mathbb{Z}_2), \{0\} )
\end{equation}
constructed in \cite[Section 3]{Alm1}. Actually, Almgren treated the case where the coefficient group is $G=\mathbb{Z}$, but as Pitts noted in \cite{P}, all the methods extend to the case $G=\mathbb{Z}_2$. There is a number $\mu>0$ such that if $T\in \mathbf{I}_k(M,\mathbb{Z}_2)$ has no boundary and $\mathcal{F}(T)\leq \mu$, then there is an $S\in \mathbf{I}_{k+1}(M,\mathbb{Z}_2)$ such that $\partial S =T$ and $$\mathbf{M}(S)=\mathcal{F}(T)=\inf\{\mathbf{M}(S') ; S'\in \mathbf{I}_{k+1}(M,\mathbb{Z}_2) \text{ and } \partial S' = T \}.$$
Such an $S$ is called an $\mathcal{F}$-isoperimetric choice for $T$. A chain map 
$$\Phi : I(1,j)\to \mathbf{I}_*(M)$$
of degree $n$ is a graded homomorphism $\Phi $ of degree $n$, such that $\partial \circ \Phi = \Phi \circ \partial$  and $\Phi(\alpha)$ is $\mathcal{F}$-isoperimetric for $\Phi(\partial \alpha)$ where $\alpha \in I(1,j)_1$. Now let $$[f]\in \pi_1(\mathcal{Z}_n(M,\mathcal{F},\mathbb{Z}_2), \{0\} )$$
be a class whose one of the representative maps is 
$$f : (I,\{0,1\}) \to (\mathcal{Z}_n(M,\mathcal{F},\mathbb{Z}_2), \{0\} ).$$
The isomorphism $F: \pi_1(\mathcal{Z}_n(M,\mathcal{F},\mathbb{Z}_2), \{0\} ) \to H_{n+1}(M,\mathbb{Z}_2)$ is defined as follows: take any integer $m$ sufficiently large, there is a chain map
$$\Phi : I(1,m)\to \mathbf{I}_*(M,\mathbb{Z}_2)$$
of degree $n$, such that 
$$\Phi([x])=f(x) \quad \forall [x] \in I(1,m)_0.$$
Then 
$$\sum_{\alpha \in I(1,m)_1} \Phi(\alpha)$$ 
is a cycle in $\mathbf{I}_{n+1}(M,\mathbb{Z}_2)$ which depends neither on $m$ if the latter is chosen large enough, nor on the representative $f$. Because the homology groups of the chain complex $\mathbf{I}_*(M,\mathbb{Z}_2)$ are isomorphic with the singular homology groups of $M$ with coefficient group $\mathbb{Z}_2$, it makes sense to set
$$F([f]) = \big[  \sum_{\alpha \in I(1,m)_1} \Phi(\alpha)  \big] \in H_{n+1}(M,\mathbb{Z}_2).$$

\subsection {} \textbf{Min-max theory in the continuous setting}

The theory of Almgren and Pitts deals with discrete sweepouts, and this can bring some technical complications when constructing explicit sweepouts. For this reason, De Lellis and Tasnady \cite{Tasnady} wrote a version of this theory in the continuous setting, based on ideas in \cite{Smith} and \cite{C&DL}. In this subsection, we recall the basic notions in this setting.

Let $(M^{n+1},g)$ be a closed Riemannian manifold. In what follows, the topological boundary of a subset of $M$ will be denoted by $\partial$. Consider an open subset $N\subset M$ whose boundary $\partial N$, when non-trivial, is a rectifiable hypersurface of finite $n$-dimensional Hausdorff measure. Suppose also that if $\partial N$ is non empty, each of its connected components $C$ separates $M$ (i.e. $M\backslash C$ has two connected components). The notation for the $m$-dimensional Hausdorff measure will be $\mathcal{H}^m$. Take $a<b\in \mathbb{R}$, $k\in \mathbb{N}$. 

\begin{definition} \label{definition}
A family of $\mathcal{H}^n$-measurable closed subsets $\{\Gamma_t\}_{tÊ\in [a,b]^k}$ in $N\cup \partial N$ with finite $\mathcal{H}^n$-measure is called a generalized smooth family if
\begin{itemize}
\item for each $t $ there is a finite subset $P_t \subset N$ such that $\Gamma_t\cap N$ is a smooth hypersurface in $N\backslash P_t$,
\item $t \mapsto \mathcal{H}^n(\Gamma_t)$ is continuous and $t \mapsto \Gamma_t$ is continous in the Hausdorff topology,
\item $\Gamma_t \to \Gamma_{t_0}$ smoothly in any compact $U\subset\subset N\backslash P_{t_0}$ as $t\to t_0$. 
\end{itemize}
If $\partial N= \varnothing$, a generalized smooth family $\{\Sigma_t\}_{t\in [a,b]}$ is called a continuous sweepout of $N$ if there exists a family of open subsets $\{\Omega_t\}_{t\in [a,b]}$ of $N$ such that
\begin{enumerate} [label=(\roman*)]
\item $(\Sigma_t\backslash \partial\Omega_t) \subset P_t$ for any $t\in(a,b]$,
\item $\mathcal{H}^{n+1}(\Omega_t\backslash \Omega_s) + \mathcal{H}^{n+1}(\Omega_s\backslash \Omega_t) \to 0$, as $s\to t$,
\item $\Omega_a =\varnothing$, and $\Omega_b= N$.
\end{enumerate}
When $\partial N \neq \varnothing$, a continuous sweepout is required to satisfy the above conditions, except that $\partial\Omega_t$ denotes the boundary of $\Omega_t$ in $N$ and (iii) is replaced by 

$(iii)$ $\Omega_b =N $, $\Sigma_a = \partial N$ and $\Sigma_t \subset N$ for $t \in (a,b]$.
\end{definition}

\begin{definition} \label{homotopic}
When $\partial N=\varnothing$, two continuous sweepouts $\{\Sigma_t^1\}_{t\in [a,b]}$ and $\{\Sigma_t^2\}_{t\in [a,b]}$ are homotopic if there is a generalized smooth family $\{\Gamma_{(s,t)}\}_{(s,t)\in [a,b]^2}$, such that $\Gamma_{(a,t)} =\Sigma_t^1$ and $\Gamma_{(b,t)} =\Sigma_t^2$. When $\partial N \neq \varnothing$, we also require the following condition: $\Gamma_{(s,t)} \subset N$ for $t\in (a,b]$ and there exists a small $\alpha>0$ such that $\Gamma_{(s,t)}=\Gamma_{(a,t)}$ for $(s,t)\in [a,b]\times [a,a+\alpha]$. 

A family $\Lambda$ of continuous sweepouts is called homotopically closed if it contains the homotopy class of each of its element.

\end{definition}

\begin{remarque}
\begin{enumerate}
\item Definitions \ref{definition} and \ref{homotopic} are adapted from the definitions in \cite{Tasnady}. Little modifications are done compared with \cite{Tasnady} and \cite{Zhou} because we have to deal with a non-smooth boundary $\partial N$.
\item A referee pointed out that Definition \ref{homotopic} (see Definition 2.5 in \cite{Tasnady}) may not be the most natural definition of homotopy classes: it may be more intuitive to impose, for each $s_0\in[a,b]$, that the intermediate family $\{\Gamma_{(s_0,t)}\}_{t \in [a,b]}$ is a sweepout as well, instead of just requiring $\{\Gamma_{(s,t)}\}_{(s,t) \in [a,b]^2}$ to be a generalized smooth family. We decided to keep the original definition, even though the proofs of existence and regularity in \cite{Tasnady} would still work for the stronger definition because all deformations they used to construct competitors preserve the homotopy class in the strong sense.
\end{enumerate}
\end{remarque}

If $\Lambda$ is a homotopically closed family of continuous sweepouts, the width of $\Lambda$ in $N$ is defined as the min-max quantity
$$W(N,\partial N, \Lambda) = \inf_{\{\Sigma_t\}\in \Lambda} \max_t \mathcal{H}^n(\Sigma_t).$$
A sequence $\{\{\Sigma_t^k\}_{t\in [a,b]}\}_{k\in \mathbb{N}} \subset \Lambda$ is called a minimizing sequence if 
$$\max_t \mathcal{H}^n(\Sigma_t^k) \to W(N,\partial N,\Lambda)  \text{  as } k\to \infty.$$
A sequence of slices $\{\Sigma_{t_k}^k\}_{k \in \mathbb{N}}$ is called a min-max sequence if $$\mathcal{H}^n(\Sigma_{t_k}^k) \to W(N,\partial N, \Lambda) \text{  as } k\to \infty.$$

\subsection {} \textbf{Some min-max constructions}

Let $(M,g)$ be a closed $(n+1)$-manifold with $2\leq n \leq 6$. For a two-sided hypersurface $\Sigma$, the convention is that the mean curvature vector is $-H(\Sigma) \nu$, where $\nu$ is a continuous choice of unit normal vector of $\Sigma$, called outward unit normal. A hypersurface with boundary will be said to be generically immersed if it is the image of an immersion with normal crossings, in other words a "generic" immersion (see Definition 3.1 in \cite{GG}).

\begin{definition} \label{wmconvex}
Let $N$ be an open subset of $M$ and suppose that $\partial N$ is a non-empty rectifiable hypersurface. The boundary $\partial N$ is said to be piecewise smooth mean convex if it satisfies the following property: 

\begin{enumerate} [label=(\roman*)]
\item there is a generically immersed compact two-sided $n$-submanifold $F$ with smooth boundary such that $N$ is a connected component of $M\backslash F$ and $$\partial N\cap \partial F =\varnothing,$$
\item $F$ has positive mean curvature at every point of $\partial N $ with respect to any outward unit normal determined by $N$. 
\end{enumerate}
More generally, a rectifiable hypersurface $A$ is said to be piecewise smooth mean convex if there is an open set $N\subset M$ such that $A$ is an open subset of $\partial N$, the first point $(i)$ is true and the second point $(ii)$ is satisfied for every point of $A$.

\end{definition}

We emphasize that (piecewise smooth) mean convexity means that the mean curvature is strictly positive. The first property of a piecewise smooth mean convex boundary is that it acts as a barrier for area minimizing problems.

\begin{prop} \label{vector field}
Let $(M,g)$ be a closed $(n+1)$-manifold and $N$ an open set of $M$ such that the boundary $\partial N$ is non-empty and piecewise smooth mean convex. Then there is a positive number $\delta>0$ and a vector field $\mathbf{V}$ in $N$ such that 
\begin{enumerate} [label=(\roman*)]
\item $\sum_{i=1}^n \langle \nabla_{e_i}\mathbf{V} , e_i \rangle \leq 0$ for every orthonormal family $\{e_1,...,e_n\}$ of vectors whose base point is in $N$, 
\item if $p\in N$ and $d(p,\partial N)=d \leq \delta$, then $$\langle \mathbf{V}(p), \frac{\partial}{\partial s}\bigg|_{s=d}   \exp_q(-s\nu) \rangle>0,$$ 
where $q$ is any point of $\partial N$ such that $d(p,q)=d$ and $\nu$ is the outward unit normal of $\partial N$ at $q$ (which is well defined).
\end{enumerate} 
Consequently, if $\Sigma$ is a hypersurface in $N$, then there is a diffeomorphism $\Psi$ of $N$ such that $\mathcal{H}^n(\Psi(\Sigma))\leq\mathcal{H}^n(\Sigma)$ and $d(\Psi(\Sigma), \partial N) \geq\delta/2$.

\end{prop}

\begin{proof}
This proposition is an extension of Lemma 2.2 in \cite{MaNe}. The idea for the piecewise smooth case is that, starting from the vector field constructed in \cite{MaNe}, we will add some extra vector fields near the non-smooth parts of $\partial N$ so that the final vector field is always pointing inward on $\partial N$. Recall that on smooth pieces of $\partial N$, the outward unit normal is well determined. By hypothesis, there is a two-sided hypersurface $F$ such that $N$ is a connected component of $M\backslash F$ and has positive mean curvature at every smooth point of $\partial N $ with respect to the outward unit normal determined by $N$. By reducing $F$ if necessary, we can suppose that the mean curvature of $F$ is positive everywhere. Let us consider an immersed compact hypersurface $S \subset M$ with boundary endowed with a continuous choice of outward unit normal vector $\nu_S$ and suppose that $S$ has positive mean curvature with respect to $\nu_S$. We can associate to $S$ a positive real number $a(S)$ such that 
$$\tilde{\exp} : [0,2a(S)]\times S \to {M}$$
$$\tilde{\exp}(r,x) = \exp_x(-r\nu_S)$$
is a local diffeomorphism, and the surface with boundary $C_r=\{r\}\times S$ has positive mean curvature. As shown in \cite{MaNe}, when $S$ is embedded there is a vector field on $\tilde{\exp} ([0,2a(S)]\times S)$, called $X(S)$, such that for all $p= \tilde{\exp}(r,x)$, $X(S)_p=\psi(r) \frac{\partial}{\partial r}$ where $\psi$ is a function positive if $r<a(S)$ and $\psi(p)=0$ if $r> a(S)$. Moreover for every orthonormal family $\{e_1,...,e_n\}$ of vectors in $\tilde{\exp} ([0,2a(S)]\times S)\cap N$, 
\begin{equation} \label{reduce}
\sum_{i=1}^n \langle \nabla_{e_i}X(S) , e_i \rangle \leq 0.
\end{equation} 
Suppose additionally that:
\begin{equation} \label{fgh}
N \cap S = \varnothingÊ\text{ and } \partial N \cap \partial S = \varnothing.
\end{equation}
The hypersurface $F$ for instance has this property. We will always suppose $a(S)$ chosen small enough so that 
\begin{equation} \label{discontinu}
N\cap \tilde{\exp} ([0,2a(S)]\times \partial S) = \varnothing.
\end{equation} 
Since the construction of $X(S)$ is local in the sense that $X(S)_p$ at $p=\tilde{\exp}(r,x)$ only depends on $r$ and a neighborhood of $x$ in $S$, we can define such a vector field $X(S)$ on $\tilde{\exp} ([0,2a(S)]\times S)$ when $S$ is merely immersed: the domain where this vector field is defined overlaps itself when $S$ is non embedded, so here at any given point in $\tilde{\exp} ([0,2a(S)]\times S)$, $X(S)$ is the sum of all the local contributions. By linearity, (\ref{reduce}) remains true. We extend $X(S)$ by $0$ outside $\tilde{\exp}([0,2a(S)]\times S)$ and the new vector field is called $Y(S)$. Let us check that this extension $Y(S)$ is well-defined and smooth in $N$. Note that when $\Omega$ is an embedded domain in $S$, we can take $a(\Omega) = a(S)$ and if we extend $X(\Omega)$ by $0$ outside $\tilde{\exp}([0,2a(S)]\times \Omega)$, the new vector field $Y(\Omega)$ can be non-smooth only on $\tilde{\exp} ([0,2a(S)]\times \partial \Omega) \cup \Omega$. For any $p\in M$, there are only finitely many points $\{x_j\}_{j\in J}\subset S$ such that $\tilde{\exp}(r_j,x_j)=p$ for some $r_j\in[0,2a(S)]$ and let $\Omega_j$ be an embedded neighborhood of $x_j$ in $S$. If moreover $p\in N$, then according to (\ref{fgh}) and (\ref{discontinu}), every $r_j$ is positive and each $x_j$ is contained in the interior of $S$, so $\Omega_j$ can be chosen relatively compact in the interior of $S$. By our definition for all $q\in N$ close to $p$, 
\begin{align*}
Y(S)_q & = 0 &\text{ if } J=\varnothing,\\
& = \sum_{j\in J} Y(\Omega_j)_q &\text{ otherwise}
\end{align*}
and from this expression it becomes clear that ${Y}(S)$ is well-defined and smooth in $N$. The restriction of this vector field to $N$ is still called $Y(S)$.

Now, we want to construct a vector field in $N$ similar to the one constructed in the proof of Theorem 2.1 in \cite{MaNe}. The vector field $Y(F)$ seems to be a good candidate but the flow associated to this vector field in $M$ perhaps "leaves" $N$ around the non-smooth parts of $\partial N$: the reader can think of the situation where locally at $x\in\partial N$, $\partial N$ is made of two half-disks intersecting with an interior angle smaller than $\pi/2$, then $Y(F)$ may not point inside $N$ at $x$ if the contributions of the two half-disks are very different. Thus we have to modify $Y(F)$ near the parts of $\partial N$ where smooth pieces intersect, so that the flow stays in $N$. For $0 \leq k\leq n-1$, let $P_k$ be the $k$-dimensional part of $F$, namely the set of points $p\in F$ locally lying in at least $n+1-k$ distinct hypersurfaces. In fact, $P_k$ is the image by an immersion $\varphi_k$ of a $k$ dimensional manifold $\tilde{P}_k$, because $F$ is generically immersed. Denote by $UM\restriction_{P_k}$ the restriction of the unit tangent bundle of $M$ to ${P_k}$. Note that since $F$ is two-sided, we can find a constant $\epsilon>0$ with this property: for all $0 \leq k\leq n$ there is a a smooth function
$$\eta_k : \tilde{P}_k \to UM\restriction_{P_k}$$
such that for all $x \in \varphi_k^{-1}((P_k\backslash P_{k-1}) \cap \partial N)$, $\eta_k(x)$ is a unit vector orthogonal to $d\varphi_k(T_x\tilde{P}_k)$ and 
\begin{equation} \label{pointing inwards}
\langle \eta_k(x), \nu \rangle>\epsilon
\end{equation}
for each outward normal $\nu$ of a smooth piece of $\partial N$ touching $\varphi_k(x)$. Let $\mathbf{inj}(M)$ be the injectivity radius of $M$. Consider $x\in \tilde{P}_k$ ($0\leq k\leq n-1$), we use momentarily the notation $p=\varphi_k(x) $. Denote by $D_k(x)$ the $(n+1-k)$-disk $$\{\exp_{p}(rv) ; r<\mathbf{inj}(M) \text{ and } v \perp d\varphi_k(T_x\tilde{P}_k) \}.$$ For $r>0$ and $0 \leq k\leq n-1$, consider 
$$S_k (r)= \{y\in D_k(x) ; x\in \tilde{P}_k, d(y,p)\leq r^2 \text{ and } d(y,\exp_p(-r\eta_k(x)))=r\}.$$
Define also $S_{n} = F$.

We can choose a sequence $(r_0,...,r_{n-1})$ of small enough positive numbers so that the following properties are satisfied:
\begin{enumerate} [label=(\roman*)]
\item for each $0 \leq k\leq n$, $S_k = S_k(r_k)$ is a smooth immersed two-sided hypersurface with boundary verifying (\ref{fgh}) and $\eta_k$ determines an outward unit normal still denoted by $\eta_k$,
\item for $0 \leq k\leq n-1$, $S_k$ has positive mean curvature with respect to $\eta_k$, 
\item the $a_k = a(S_k)$ ($0\leq k\leq n$) are chosen so that we can find positive numbers $b_k$ ($0\leq k\leq n$) satisfying the following conditions. $Y(S_0)$ points inwards on $\partial N$ and points strictly inwards on $\{y\in \partial N ; d(y,P_0) < b_0\}$; for $0 \leq k\leq n-1$, $Y(S_{k+1})$ points inwards on $ \{y \in \partial N; d(y,P_k) > b_k/2\}$ and points strictly inwards on $\{y \in \partial N; d(y,P_k) > b_k/2, d(y,P_{k+1})<b_{k+1}\}$.
\end{enumerate} 
We briefly justify these properties. Observe that Definition \ref{wmconvex} implies: 
\begin{itemize}
\item for $0 \leq k\leq n$, $\varphi_k(\partial \tilde{P}_k) \cap (N\cup \partial N) = \varnothing$, 
\item at any point $q\in (P_k\backslash P_{k-1}) \cap \partial N$, there is a diffeomorphism $\mathcal{D} :  \mathbb{R}^{n+1} \to B(q,r_q)$ such that 
$$\mathcal{D}(\bigcup_{i=1}^{n+1-k} \{ x_i=0\}) = B(q,r_q) \cap F$$
provided $r_q$ is small,
\item $B(q,r_q)\cap N$ is exactly one of the connected components of $B(q,r_q)\backslash F$ provided $r_q$ is small.
\end{itemize}
Item (i) follows from that observation and from \ref{pointing inwards}. For $0\leq k \leq n-1$ and $r_k$ very small, $S_k$ looks like part of a thin "k-tube" pasted along $P_k$ which has bounded curvature, so $S_k$ has large mean curvature as claimed in (ii). Since $S_0$ is a finite union of subsets of little n-spheres lying on the vertices of $\partial N$ and curved toward $-\eta_0$, if $a(S_0)$ is small, $Y(S_0)$ indeed points inwards on $\partial N$ and strictly inwards in a certain neighborhood $\{y\in \partial N ; d(y,P_0) < b_0\}$ of $P_0$. Then, we choose $a(S_1)$ small so that 
$$\tilde{\exp} ([0,2a(S_1)]\times S_1) \cap \{y\in N ; d(y,P_0) > b_0/2\}$$
does not overlap: the argument becomes local and by reducing $a(S_1)$ if necessary, $Y(S_1)$ points inwards on $\{y\in \partial N ; d(y,P_0) > b_0/2\}$ and strictly inwards on a neighborhood $$\{y \in \partial N; d(y,P_0) > b_0/2, d(y,P_{1})<b_{1}\}$$ of $P_1\cap \{y\in \partial N ; d(y,P_0) > b_0/2\}$. We continue by induction to check (iii) (see Figure (\ref{illustratif})).

\begin{figure} 
\hbox{\hspace{-20mm}
\includegraphics[scale=0.6]{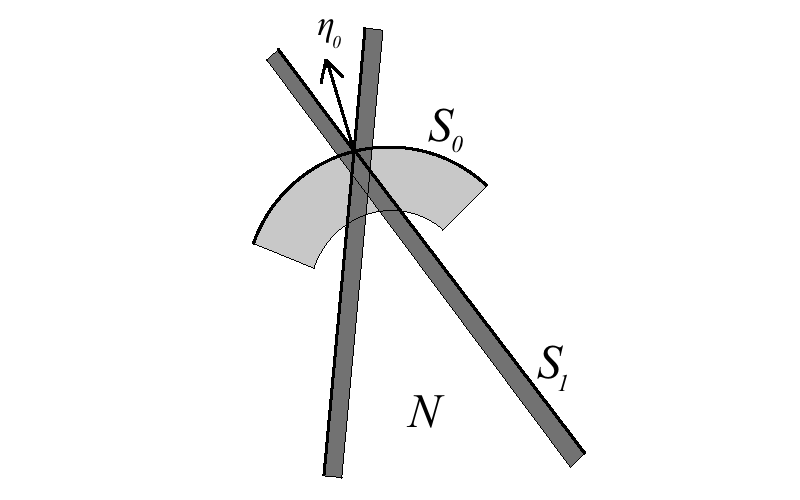}}
\caption{The pale grey region represents $\{Y(S_0)>0\}$, whereas the deeper grey region corresponds to $\{Y(S_1)>0\}$.} The latter overlaps in $\{Y(S_0)>0\}$ but not in $N\backslash \{Y(S_0)>0\}$.
\label{illustratif}
\end{figure}

For positive constants $\lambda_0, ...,\lambda_{n-1}$ chosen later, set
$$\mathbf{V} = Y(S_n) + \sum_{j=0}^{n-1} \lambda_iY(S_j).$$
By linearity, this vector field automatically satisfies $\sum_{i=1}^n \langle \nabla_{e_i}\mathbf{V} , e_i \rangle \leq 0$ for every orthonormal family $\{e_1,...,e_n\}$ of vectors whose base point is in $N$. Then we choose $\lambda_{n-1}$, ..., $\lambda_{0}$ in this order, in the following way: first take $\lambda_{n-1}$ very large so that $Y(S_n) + \lambda_{n-1} Y(S_{n-1} )$ points strictly inwards at least on $ \{x \in \partial N ; d(x,P_{n-2})> b_{n-2}/2 \}$. This is possible because of point (iii) of the previous properties. Then we choose $\lambda_{n-2}$ so that $Y(S_n) + \lambda_{n-1} Y(S_{n-1} ) +  \lambda_{n-2} Y(S_{n-2} )$ points strictly inwards at least on $ \{x \in \partial N ; d(x,P_{n-3})> b_{n-3}/2 \}$. We continue until $\lambda_0$, which can be chosen large enough so that $\mathbf{V}= Y(S_n) + \sum_{j=0}^{n-1} \lambda_iY(S_j)$ points strictly inwards everywhere on $\partial N$. In this way, there is a positive constant $\delta$ such that if $p\in N$ and $d(p,F)=d \leq \delta$, then $$\langle \mathbf{V}(p), \frac{\partial}{\partial s} \bigg|_{s=d} \exp_q(-s\nu)\rangle>0,$$ 
where $q$ is any point of $F$ such that $d(p,q)=d$ and $\nu$ is the outward unit normal of $F$ at $q$.
This proves the proposition.

\end{proof}

Thanks to the previous proposition, we can check the following version of the min-max theorem.

\begin{theo}  \label{piecew}

Let $(M,g)$ be a closed $(n+1)$-manifold with with $2\leq n \leq 6$, and $N$ an open set of $M$ possibly with a non empty topological boundary. When $\partial N \neq \varnothing$, assume that $\partial N$ is piecewise smooth mean convex. Then for any homotopically closed family $\Lambda$ of sweepouts of $N$, with $$W(N,\partial N,\Lambda)>\mathcal{H}^{n}(\partial N)$$ when $\partial N \neq \varnothing$, there exists a min-max sequence $\{\Sigma_{t_k}^k\}$ of $\Lambda$ converging in the varifold sense to an embedded minimal hypersurface $\Sigma$ (possibly disconnected), contained in $N$. Moreover the $n$-volume of $\Sigma$, if counted with multiplicities, is equal to $W(N,\partial N,\Lambda)$. 

\end{theo}

\begin{proof}

In \cite{Zhou}, this theorem is proved for a smooth mean convex boundary $\partial N$ (see Theorem $2.7$ in \cite{Zhou}). Here, we can use the previous proposition to study the piecewise smooth case. Recall that the proof in \cite{Zhou} uses an idea of \cite{MaNe}, where the authors construct a vector field in $N$ only nonvanishing in a small neighborhood of $\partial N$ and, using the associated flow, find $a>0$ and a minimizing sequence of sweepouts $\{\{\Sigma_t^k\}_{t\in [0,1]}\}_k$ such that
$$\mathcal{H}^n(\Sigma^k_t) \geq W(N,\partial N,\Lambda) - \delta_1 \Rightarrow d(\Sigma^k_t,\partial N) \geq \delta/2,$$
where $\delta_1 = \frac{1}{4}(W(N,\partial N, \Lambda) - \mathcal{H}^n(\partial N)) >0$ and $d$ is the distance function in $M$. Consider the area-decreasing vector field $\mathbf{V}$ constructed in Proposition \ref{vector field}, then the flow $\Phi_t$ associated to $\mathbf{V}$ sends $N$ into $N$ for all times and $\Phi_t(N) \subset \{p \in N; d(p,\partial N)\geq \delta/2 \}$ for sufficiently large times. These properties of $\mathbf{V}$ allow us to conclude as in \cite[Theorem 2.1]{MaNe} and \cite[Theorem 2.7]{Zhou}.

\end{proof}

We will say that a hypersurface $\Sigma$ is produced by Amgren-Pitts' theory with the fundamental class of $H_{n+1}(M,\mathbb{Z}_2)\simeq\mathbb{Z}_2$ if the varifold $|\Sigma|$ belongs to the critical set $\mathbf{C}(S)$ of a sequence $S\in \Pi$, where $\Pi\in \pi^\sharp_1(\mathcal{Z}_n(M,\mathbf{M},\mathbb{Z}_2),\{0\})$ corresponds to the fundamental class of $H_{n+1}(M,\mathbb{Z}_2)$ by the isomorphism (\ref{iso}) and Remark \ref{isohomo}. 
The following proposition will be useful. 

\begin{prop} \label{sweepout}
Let $\Sigma$ be an embedded connected unstable minimal hypersurface in $M$ and suppose that there is no embedded stable minimal hypersurface $S$ with $\mathcal{H}^n(S) < \mathcal{H}^n(\Sigma)$. Then there is a minimal hypersurface $\Gamma$ produced by Almgren-Pitts' theory with the fundamental class of $H_{n+1}(M,\mathbb{Z}_2)$ such that 
$$\mathcal{H}^n(\Gamma) \leq \mathcal{H}^n(\Sigma)$$
and if equality holds, then $\Sigma$ itself is produced by Almgren-Pitts' theory with the fundamental class of $H_{n+1}(M,\mathbb{Z}_2)$.
\end{prop}

\begin{proof}
By intersection theory, if $\Sigma$ does not separate, we can minimize the area in its non-trivial homology class in $H_n(M,\mathbb{Z}_2)$ and by the regularity theory for area minimizing flat chains modulo $2$ (\cite[Theorem 2.4]{Morgan}) we obtain a stable non-trivial minimal hypersurface $S$ with $\mathcal{H}^n(S) < \mathcal{H}^n(\Sigma)$. So actually $\Sigma$ is two-sided. The procedure to produce $\Gamma$ is then standard (\cite{MaNe},\cite{Zhou}). One constructs a continuous sweepout $\{\Sigma_t \}_{[0,1]}$ of $M$ as in \cite[Proposition 3.6]{Zhou} (whose proof does not use the orientability of $M$ but only the fact that the hypersurface $\Sigma$ is two-sided): 
\begin{enumerate} 
\item $\Sigma_{1/2}=\Sigma$,
\item $\mathcal{H}^n(\Sigma_t) \leq \mathcal{H}^n(\Sigma)$ with equality only if $t =1/2$,
\item $\{\Sigma_t\}_{t\in [1/2-\epsilon,1/2+\epsilon]}$ forms a foliation of a neighborhood of $\Sigma$.
\end{enumerate}
$\{\Sigma_t \}$ determines by interpolation (\cite[Theorem 5.5]{Zhou} and Remark \ref{modulo}) a homotopy sequence of mappings $S= \{\phi_i\}_{i\in \mathbb{N}}$ where
$$\phi_i : I(1,n_i)_0 \to \mathcal{Z}_n(M,\mathbf{M},\mathbb{Z}_2) \text{ with fineness $\delta_i$, } \quad \lim\limits_{i\to \infty} \delta_i= 0,  \lim\limits_{i\to \infty} n_i = \infty,$$
$$\text{ and } \mathbf{L}(S)\leq \mathcal{H}^n(\Sigma).$$ 
Let $\Pi$ be the homotopy class of mappings of $S$. By \cite[Theorem 5.8, Claim 3]{Zhou} and Remark \ref{modulo}, $\Pi$ corresponds to the fundamental class of $H_{n+1}(M,\mathbb{Z}_2)$. By \cite[Theorem 4.10, Theorem 7.12]{P} and Remark 6, there is a smooth minimal hypersurface $\Gamma$ with $$\mathcal{H}^n(\Gamma)=\mathbf{L}(\Pi) \leq \mathbf{L}(S) \leq \mathcal{H}^n(\Sigma).$$
Suppose that $\mathbf{L}(\Pi) = \mathcal{H}^n(\Sigma)$, which means that $S$ is a critical sequence. We can consider a sequence of slices $\phi_i(\alpha^i)$, where $  \alpha^i=[x_i] \in I(1,n_i)_0$, such that $$\mathbf{M}(\phi_i(\alpha^i)) \to \mathbf{L}(\Pi).$$ 
Because of \cite[Theorem 5.5 (1)]{Zhou} and Remark \ref{modulo}, necessarily $x_i\to 1/2$ so $\phi_i(\alpha^i)$ converge to $\Sigma_{1/2}=\Sigma$ in the flat topology by \cite[Theorem 5.5 (3)]{Zhou}. It is known that if $T_j \in \mathcal{Z}_k(M,\mathbb{Z}_2)$ converges to $T \in \mathcal{Z}_k(M,\mathbb{Z}_2)$ in the flat topology and the sequence of varifolds $|T_j|$ converge to $V \in \mathcal{V}_k(M)$, then $||V||(M) = \mathbf{M}(T)$ implies that $V=|T|$ (see \cite[Chapter 2, 2.1, (18), (f)]{P}). It follows that if $$\lim_i |\phi_i(\alpha^i)| = V \text{ and } ||V||(M) = \mathbf{L}(\Pi)= \mathcal{H}^n(\Sigma)$$ then $V=|\Sigma|$. Thus the only element of the critical set $\mathbf{C}(S)$ is $|\Sigma|$. This shows that $\Sigma$ itself is produced by Almgren-Pitts' theory with the fundamental class of $H_{n+1}(M,\mathbb{Z}_2)$.

\end{proof}

\subsection{} \textbf{Outline of the proof of Theorem \ref{principal}}
   
$M$ is a closed $(n+1)$-manifold $M$ with $2\leq n \leq 6$. The main step of the proof is to show that if a least area hypersurface exists, then it is necessarily embedded (Proposition \ref{embed}). This step is essentially an extension of the idea in \cite{CC} to higher dimensions. Given a minimal hypersurface $\Sigma$ which is not embedded and with no stable embedded minimal hypersurface of less area, we construct an "optimal" non-trivial sweepout corresponding to $\Sigma$ in the sense that $\Sigma$ is the middle slice of a mapping $A :I \to \mathcal{Z}_n(M,\mathbb{Z}_2)$ continuous in the flat topology and other slices have area strictly less than $\mathcal{H}^n(\Sigma)$. This way of thinking $\Sigma$ as the middle slice of an optimal sweepout in order to compare it with other sweepouts and deduce properties about the embeddedness, the index, the area or the multiplicity of $\Sigma$ appears in \cite{CC}, \cite{MaNe}, \cite{Zhou} and \cite{MR}. We will first show that there is a partition of the family of connected components of $M\backslash \Sigma$ into two classes $\mathcal{C}_1$ and $\mathcal{C}_2$, such that 
$$\Sigma =  \bigcup\limits_{c \in \mathcal{C}_i}\partial c \text{ for } i\in\{1,2\}.$$
Roughly speaking, $A$ is then obtained by retracting $\Sigma$ to $\varnothing$ in two different ways corresponding to $\mathcal{C}_1$ and $\mathcal{C}_2$. The construction of this $1$-parameter family of currents is based on Theorem \ref{piecew}. Then we use the interpolation proved by Marques and Neves in \cite{MaNeWillmore} for $n=2$, later checked for higher dimensions by Zhou \cite{Zhou}. It enables to obtain, from a family of currents continous in the flat topology, a homotopy sequence of mappings, to which one can apply Pitts' theory with some precise control on the obtained discrete slices. This interpolation theorem applied to $A$ will give the wanted $S$. The latter belongs to a homotopy class of mappings into $(\mathcal{Z}_n(M,\mathbf{M},\mathbb{Z}_2),\{0\})$ called $\Pi$ such that 
$$0<\mathbf{L}(\Pi) < \mathcal{H}^2(\Sigma).$$ 
The theory of Almgren and Pitts then produces an embedded minimal hypersurface whose area is strictly less than $\mathcal{H}^n(\Sigma)$, which completes the main step. In these arguments, we use flat chains modulo $2$ (see Remark \ref{modulo}) because we work with non necessarily orientable submanifolds. 

To illustrate what are the two ways of retracting $\Sigma$, let's consider the following $1$-parameter family of currents $\varphi : [-1,1] \to \mathcal{Z}_n(\mathbb{R}^{n+1},\mathbb{Z}_2)$:
$$ \varphi(t) = \left\{ \begin{array}{rcl}
 \partial\{x_1\leq t,x_2\leq t\} \cup \partial \{ x_1\geq-t,x_2\geq -t\}   & \mbox{for}
& t<0\\ \partial \{x_1\geq t,x_2\leq -t\} \cup \partial \{x_1 \leq -t, x_2\geq t\} & \mbox{for} & t \geq 0
\end{array}\right.$$
where $x_1,x_2,...,x_{n+1}$ are the coordinate functions of $\mathbb{R}^{n+1}$. Let's describe what is happening in the plane generated by $x_1$ and $x_2$. During the first half $[-1,0[$, the boundaries of the South-West and North-East corners get closer until they meet along the "edge" $\{x=0,y=0\}$ and during the second half $[0,1]$, the boundaries of the North-West and South-East corners move away one from another. This family of currents will be the local model for the construction of the deformation described previously. 
 
When one tries to extend the proof in \cite{CC} to $n\geq 2$, several technical issues arise in higher dimension. Firstly, Calabi and Cao only had to deal with curves which form an eight, whereas here $\Sigma$ is not a priori so nicely immersed and we have to understand why there is still a good partition of the components of $M\backslash \Sigma$ into two classes (see Lemma \ref{partition}). Secondly in order to construct the optimal sweepout, we have to use the technics in \cite{MaNe} and \cite{Zhou}, but they are only developed for the smooth embedded case and here, we have to work with components of $M\backslash \Sigma$ whose boundaries are only rectifiable (see Theorem \ref{piecew} and Proposition \ref{claim}).

\section{Proof of Theorem \ref{principal}} \label{proof}
All minimal hypersurfaces considered here are closed. Define the following quantities:
$$\mathfrak{A}(M)= \inf \{\mathcal{H}^n(\Sigma') ; \Sigma' \subset M \text{ is a minimal hypersurface }\},$$
$$\mathfrak{A}_S(M) = \inf\{\mathcal{H}^n(\Sigma') ; \Sigma' \subset M \text{ is minimal, stable and embedded }\}.$$
By definition, $\mathfrak{A}_S(M)\geq \mathfrak{A}(M)$. In contrast to \cite{Zhou} and \cite{MR}, we will not take care of the orientability, thus here "of least area" is to be understood in the geometric sense. We will prove the following more precise version of Theorem \ref{principal}:

\begin{theo} \label{principal bis}
Let $M$ be a closed $(n+1)$-manifold where $2\leq n\leq 6$. Then there exists a least area minimal hypersurface $\Sigma_0$, i.e. $$\mathcal{H}^n(\Sigma_0)=\mathfrak{A}(M).$$
Such a hypersurface is either an embedded stable minimal hypersurface or a two-sided hypersurface of index one produced by the min-max theory of Almgren-Pitts with the fundamental class of $H_{n+1}(M,\mathbb{Z}_2)$. In particular it is always embedded.

\end{theo}

Let $M$ be a closed $(n+1)$-manifold where $2\leq n\leq 6$. Let $\Sigma = \psi(\Gamma)$ be a minimal hypersurface, image of a closed $n$-manifold $\Gamma$ by the immersion $\psi$. The latter will be supposed to be without "double cover", i.e. there is not a pair $(U_1,U_2)$ of disjoint open sets in $\Gamma$ such that $\psi(U_1)=\psi(U_2)$; since $\Sigma$ is a minimal hypersurface, it is always possible to choose $\psi$ without double cover. Denote by $\mathcal{C}_\Sigma$ the set of connected components of $M \backslash \Sigma$. The complement in $M$ of an immersed hypersurface may be quite complicated in general but it is well described if the immersion is a map with normal crossings (see \cite[Definition 3.1]{GG}), that is, a "generic" immersion. 

Define the closed set
\begin{equation} \label{self-intersection}
E=\{x \in \Sigma;  \psi^{-1}(\{x\}) \text{ has at least two distinct elements}\}
\end{equation}
to be the set along which $\Sigma$ self-intersects. Two minimal hypersurfaces tangentially intersecting at a point can be locally written as graphs of two functions whose difference satisfies a homogeneous elliptic equation. The explicit equation is derived in \cite[Chapter 7, \S 1]{Acourseinminimalsurfaces} for instance (the authors do it for $n+1=3$ but an analogue equation is clearly true for higher dimensions). By the description of nodal sets for elliptic equations in \cite{Aron} or \cite[Theorem 1.10]{HS}, it is known that $E$ is an $(n-1)$-rectifiable set: more precisely, the set where $\Sigma$ intersects itself tangentially is an $(n-2)$-rectifiable set.

\begin{lemme} \label{partition}
Let $\Sigma$ be a non-embedded minimal hypersurface in $M$ such that $\mathcal{H}^n(\Sigma) \leq \mathfrak{A}_S(M) $. Then there is a partition of $ {\mathcal{C}}_\Sigma$ into $2$ classes ${\mathcal{C}}_1$ and ${\mathcal{C}}_2$ such that if ${c}_1$, ${c}_2$ are in $ {\mathcal{C}}_\Sigma$, if for $p \in \Sigma\backslash E$ and $r>0$:
$$(B(p,{r})\backslash \Sigma) \subset ({c}_1\cup {c}_2),$$
then $c_1$ and $c_2$ are not in the same class.
In particular, for $i\in \{1,2\}$, $$\mathcal{H}^n(\Sigma)  =  \sum\limits_{c \in {\mathcal{C}_i}}\mathcal{H}^n(\partial {c}).$$
Moreover, such a partition is uniquely determined.

\end{lemme}

\begin{proof}
Let's suppose that there is no embedded stable minimal hypersurface with area strictly less than $\mathcal{H}^n(\Sigma)$. We want to construct a partition of $ {\mathcal{C}}_\Sigma$ satisfying the described property. Take a point $p$ in an open connected component of $M\backslash \Sigma$. For all $c\in {\mathcal{C}}_\Sigma$, choose a smooth curve $\gamma_c: [0,1]$ such that $\gamma_c(0) = p$, $\gamma_c(1)$ is in the interior of $c$ and $\gamma$ is generic in the sense that it only intersects $\Sigma$ at a finite number of points in $\Sigma\backslash E$ and it does so transversally (this choice is possible since $E$ is a closed set whose Hausdorff dimension is $n-1$). Define the class of $c$ as follows: if $\gamma_c$ intersects $\Sigma$ an odd number of times then declare $c$ to be in ${\mathcal{C}}_1$, otherwise declare it to be in ${\mathcal{C}}_2$. It remains to show that this method gives the right partition. It suffices to show that the result of this algorithm does not depend on the choice of the pathes $\gamma_c$. Let's argue by contradiction and suppose that there are two generic paths $\mu_1,\mu_2 : I \to M$ connecting $p$ to a point $q$ in the interior of a component $c \in {\mathcal{C}}_\Sigma$, but intersecting $\Sigma\backslash E$ a different number of times modulo $2$. Then we can glue $\mu_1$ and $\mu_2$ at $p$ and $q$ so that we obtain a cycle intersecting generically $\Sigma$ an odd number of time. It means by intersection theory that $\Sigma$ represents a non-trivial class in $H_n(M,\mathbb{Z}_2)$. One can minimize the area in this homology class and eventually obtain an embedded minimal hypersurface $\Sigma'$ with $\mathcal{H}^n(\Sigma') \leq \mathcal{H}^n(\Sigma)$ by \cite[Theorem 2.4]{Morgan}. But by the assumption at the beginning of the proof, equality holds. Hence, $\Sigma$ also minimizes the area in its homology class so should be embedded by \cite{Morgan}. This is absurd, consequently the procedure gives a good partition. Finally, it remains to check the uniqueness of such a partition (by renaming the classes if necessary). If $p_i\in M\backslash \Sigma$ and $c(p_i)$ is its connected component ($i=1,2$), then there is a path $\gamma$ linking $p_1$ to $p_2$ and only intersecting $\Sigma$ a finite number of times, transversally, and away from $E$. The uniqueness then comes from the fact that the class of $c(p_1)$ determines the class of each component encountered by $\gamma$, in particular $c(p_2)$.

\end{proof}

If $\psi :\Gamma\to M$ is a minimal immersion of an $n$-dimensional manifold into $(M,g)$ such that $\psi(\Gamma)$ is two-sided, the Jacobi operator is given by
$$L\phi = \Delta\phi  + |\mathbf{A}|^2\phi+ \Ric_g(\nu,\nu)\phi,$$
where $\phi \in C^\infty(\Gamma)$, $\mathbf{A}$ is the second fundamental form, and $\nu$ is a choice of outward unit normal of $\psi(\Gamma)$ defined on $\Gamma$. We will adopt the convention that $\lambda$ is an eigenvalue of $L$ if there exists a non-zero function $\phi$ such that $L\phi + \lambda\phi=0$. Moreover, if $f\in C^\infty(\Gamma)$, define the map
$$\tilde{\exp}_{\psi,f} : \Gamma \to M$$
$$\tilde{\exp}_{\psi,f}(x) = \exp_{{\psi}(x)}(f(x)\nu(x)). $$
When $||f||_\infty$ is small, $\tilde{\exp}_{{\psi},f} $ is an immersion. Besides there is a unique choice of unit normal on 
$$\tilde{\exp}_{{\psi},rf}(\Gamma)\qquad r\in[0,1]$$
which is continous with respect to $r\in[0,1]$ and coinciding with $\nu$ at $r=0$. Thus $\tilde{\exp}_{{\psi},f} ({\Gamma} )$ is endowed with a natural choice of outward unit normal coming from $\nu$ and still called $\nu$. It is known that at each point of $\Gamma$:
\begin{equation} \label{derivative} 
\frac{\partial}{\partial r}\bigg|_{r=0} \langle \vec{H}(\tilde{\exp}_{\psi,rf}(\Gamma)), \nu \rangle   = L(f),
\end{equation}
where $\nu$ denote the natural choice of outward unit normal on $\tilde{\exp}_{\psi,rf}(\Gamma)$ coming from $\nu$.

Let $\psi :\Gamma\to M$ be an immersion of a connected $n$-dimensional manifold $\Gamma$ into $(M,g)$. If $\psi(\Gamma)$ is two-sided, then let $\tilde{\Gamma}=\Gamma_1\cup \Gamma_2$ be two copies of $\Gamma$ and define $\nu:\tilde{\Gamma} \to TM$ to be a continuous choice of an outward unit normal such that $\nu$ restricted to $\Gamma_1$ gives the opposite choice of $\nu$ restricted to $\Gamma_2$. If $\psi(\Gamma)$ is one-sided, then let $\tilde{\Gamma}$ be a connected double cover of $\Gamma$ such that there exists a continuous choice of outward unit normal  $\nu:\tilde{\Gamma} \to TM$. Denote by $\pi : \tilde{\Gamma} \to \Gamma$ the canonical projection. A function $\phi$ defined on $M$ (or $\Gamma$, or $\otimes^p T\Gamma$) lifts to a function on $\tilde{\Gamma}$ (or $\tilde{\Gamma}$, or $\otimes^pT\tilde{\Gamma}$) still denoted by $\phi$. Consider the immersion $\tilde{\psi} = \psi \circ \pi$. 

Usually if one considers a smooth two-sided unstable minimal hypersurface, one can push it using the first eigenvalue of the Jacobi operator to get a mean convex boundary which acts as a barrier for the Plateau problem for instance. Now if $\Sigma$ is a non-embedded minimal surface, and $c\in \mathcal{C}_\Sigma$, then constructing such a barrier "approximating " $\partial c$ inside $c$ is still possible. To make this statement rigorous, we begin with the following lemma which is a trick showing the existence of a hypersurface being "close" to $\Sigma$ and mean convex except perhaps in a small ball.

\begin{lemme} \label{justeacote}
Consider $\Gamma$ a compact connected $n$-dimensional manifold with a possibly non-empty smooth boundary and let $\psi :\Gamma\to M$ be a minimal immersion into $(M,g)$. Let $p\in\psi(\Gamma)$ be a point such that $\psi(\Gamma)$ is an embedded hypersurface in a neighborhood $U$ of $p$. Then, using the notations previously defined, there is a metric $h$ equal to $g$ in $M\backslash U$ and a function $$f \in C^\infty(\tilde{\Gamma} \cup\partial \tilde{\Gamma})$$ positive on $\tilde{\Gamma}$ and vanishing on $\partial\tilde{\Gamma}$  such that for all $s\in(0,1]$, $\tilde{\exp}_{\tilde{\psi},sf}$ is an immersion and the mean curvature of $\tilde{\exp}_{\tilde{\psi},sf}(\tilde{\Gamma})$, endowed with the natural choice of outward unit normal, is negative with respect to $h$.

\end{lemme}

\begin{proof}
Let $\lambda$ be the lowest eigenvalue of the Jacobi operator $L$ of $\tilde{\Gamma}$ (with Dirichlet condition at the boundary if $\partial \tilde{\Gamma}\neq \varnothing$) endowed with the outward unit normal $\nu$. We have 
\begin{align*}
\lambda & = \inf \int_{\tilde{\Gamma}} |\nabla u|^2-(|\mathbf{A}|^2+\Ric(\nu,\nu))u^2 dvol_g
\end{align*}
where the infimum is taken among all functions $u\in H^1(\tilde{\Gamma})$ ($H^1_0(\tilde{\Gamma})$ if $\partial \tilde{\Gamma}\neq \varnothing$) of $L^2$-norm one. We argue that by pertubating the metric in $U$ if necessary, we can make $\tilde{\Gamma}$ unstable. More precisely, we will construct a metric $h$ coinciding with $g$ outside $U$ such that $\psi$ is still minimal and the lowest eigenvalue for the Jacobi operator computed with $h$, called $\lambda_1$, is negative. Define the conformally changed new metric $h = \exp(2\varphi) g$, where $\varphi\in C^\infty(M)$ will be defined later. Let $\mathbf{A}_h$ denote the second fundamental form with respect to $h$. One can check the following formula:
\begin{align} \label{conformal change}
\mathbf{A}_h(a,b) = \exp(\varphi)(\mathbf{A}(a,b)+g(a,b)d\varphi(\nu)) \quad \forall a,b \in T\tilde{\Gamma}.
\end{align}
Since $\psi(\Gamma)$ is embedded in the neighborhood $U$ of $p$, it is not difficult to see that there exist a sequence of radii $\{r_i\}$ converging to zero and diffeomorphisms $\Phi_i:  B_{\mathbb{R}^{n+1}}(0,3) \to B(p,r_i)\subset M$ such that 
\begin{equation} \label{plaaa}
\frac{1}{r_i^2} \Phi_i^*g \xrightarrow[{i\to \infty}]{} \frac{1}{9}g_{st},
\end{equation}
\begin{equation} \label{colle}
\begin{split} 
\Phi_i(\mathbb{R}^n\cap B_{\mathbb{R}^{n+1}}(0,3)) & =  \psi(\Gamma)\cap B(p,r_i) \text{ and } \\
(\Phi_i^*g)_x(\frac{\partial}{\partial x_k},\frac{\partial}{\partial x_{n+1}})=0  \quad \forall x\in & \mathbb{R}^n\cap B_{\mathbb{R}^{n+1}}(0,3), \forall k\in \{1,...,n\}.
\end{split}
\end{equation}
in the $C^2$ topology, where $\mathbb{R}^n$ denotes the subset $\{(x_1,...,x_n,0)\} \subset \mathbb{R}^{n+1}$ and $g_{st}$ is the standard metric on $\mathbb{R}^{n+1}$. Consider the function 
$$\varrho: B_{\mathbb{R}^{n+1}}(0,3)\to \mathbb{R}$$
$$\varrho(x)=\frac{2}{1+|x|_{st}^2}$$
where $|.|_{st}$ denotes the standard norm on $\mathbb{R}^{n+1}$. Using the previous charts, we can define a conformally changed metric $h_i=\exp(2\phi_i)g$ for each $i$ by choosing a smooth function $\phi_i$ vanishing outside $B(p,r_i)$ and by imposing:
$$\phi_i=(\Phi_i^{-1})^*\log(\varrho) \text{ on } \Phi_i(B_{\mathbb{R}^{n+1}}(0,2)),$$
$$d\phi_i(\nu) =0 \text{ on } \psi(\Gamma)\cap B(p,r_i).$$
By (\ref{colle}) and (\ref{conformal change}), the previous two conditions are consistent and $\psi$ remains a minimal immersion into $M$ with respect to $h_i$ for all $i$.
Note that $\varrho^2g_{st}$ on $B_{\mathbb{R}^{n+1}}(0,2)$ is the metric of constant curvature one (then $B_{\mathbb{R}^{n+1}}(0,1)$ corresponds to one hemisphere of the unit $(n+1)$-sphere) and $\mathbb{R}^n\cap B_{\mathbb{R}^{n+1}}(0,2)$ is an unstable minimal hypersurface with boundary for this metric. Hence, since by (\ref{plaaa})
$$\frac{1}{r_i^2}\Phi_i^*h_i = \frac{\varrho^2}{r_i^2}  \Phi_i^*g \xrightarrow[{i\to \infty}]{C^2} \frac{\varrho^2}{9} g_{st} \text{ on } B_{\mathbb{R}^{n+1}}(0,2),$$
if we take $\varphi = \phi_i$ for $i$ sufficiently large, then $\tilde{\Gamma}$ contains an unstable hypersurface with boundary for $h=\exp(2\varphi) g,$ so it is unstable itself with respect to $h$ and $\lambda_1<0$.

Now using (\ref{derivative}), we can take a positive eigenfunction $f$ associated to $\lambda_1$ which is in $C^\infty(\tilde{\Gamma} \cup\partial \tilde{\Gamma})$ by \cite{Nirenberg} and with sufficiently small $L^\infty$ norm so that for all $s\in (0,1]$, each hypersurface $\tilde{\exp}_{\tilde{\psi},sf}(\tilde{\Gamma})$ is immersed and has negative mean curvature for $h$. This finishes the proof.

\end{proof}

Recall the following definition, introduced in \cite{MaNeWillmore}. If $\Phi:[a,b]\to \mathcal{Z}_n(M,\mathbb{Z}_2)$ is continuous in the flat topology, we define
$$\mathbf{m}(\Phi,r) = \sup\{||\Phi(x)||(B_r(p)) ; x \in [a,b], p\in M\}.$$
We will say that $\Phi$ satisfies the technical condition $[*]$ if 
$$\mathbf{m}(\Phi,r)\to 0 \text{ as } r\to 0.$$
In what follows, we consider an immersed minimal hypersurface $\Sigma=\psi(\Gamma)$, where $\Gamma$ is a closed $n$-manifold. $\Sigma$ will be said to be connected if $\Gamma$ is connected. From the earlier description of the set $E$ (see (\ref{self-intersection})), we know that the boundary of a component $c\in \mathcal{C}_\Sigma$ has positive finite $n$-dimensional Hausdorff measure and is $\mathcal{H}^n$-almost everywhere locally an embedded hypersurface. We will restrict our attention to components $c\in \mathcal{C}_\Sigma$ verifying the following "local separation" property.
$$
\mathbf{(LS)  } \quad \text{ For any } p \in \partial c \backslash E \text{ and } r>0,
B(p,{r})\backslash \Sigma \text{ is not included in } c.
$$
Note that if $\Sigma$ is non-embedded and $\mathcal{H}^n(\Sigma) \leq \mathfrak{A}_S(M) $, then any $c\in \mathcal{C}_\Sigma$ automatically satisfies condition \textbf{(LS)} by Lemma \ref{partition}.
For such a component, the next lemma enables to get rid of a certain subset of $c$ and decrease the area of $\partial c$ inside $c$. It is essential that $\Sigma$ is not embedded. We say that $x\in \partial c$ has a {\it local roof structure} if there is a ball $B(x,r_x)$ and a diffeomorphism $\mathcal{D} : \mathbb{R}^{n+1} \to B(x,r_x)$ such that 
\begin{align*}
&\mathcal{D}(\{x_i=0\})  \subset \Sigma \quad \forall i\in\{ 1,2\} \\
&\text{and }\mathcal{D}(\{x_1>0 , x_2>0\})  \subset c.
\end{align*}

\begin{lemme} \label{Plateau}
Suppose that $\Sigma$ is a non-embedded connected minimal hypersurface in $M$. Consider a component $c\in \mathcal{C}_\Sigma$ satisfying condition \textbf{(LS)}. Then there is a point $x\in \partial c$ having a local roof structure.

Moreover there is a map $\xi : [0,1] \to \mathcal{Z}_n(M,\mathbb{Z}_2)$ continuous in the flat topology such that for all $s\in[0,1]$, $\xi(s) = \partial [|G_s|]$, where $G_s$ are open sets with the following properties:
\begin{enumerate} [label=(\roman*)] 
\item$G_0=c$ and $\forall s\in[0,1]$ $G_s\subset c$, $\mathbf{M}(\xi(s))=\mathcal{H}^n(\partial G_s)$,
\item $\partial G_1 \cap c$ separates $c$ and is piecewise smooth mean convex with respect to the normal pointing outside $G_1$, 

\item $\forall s\in(0,1]$, $\mathcal{H}^n(\partial G_s) < \mathcal{H}^n(\partial c)$,
\item $\xi$ satisfies $[*]$.
\end{enumerate}

\end{lemme}
 
\begin{proof}
$\Sigma$ being connected, $\partial c$ is not smooth. One knows that the set $E$ where $\Sigma$ self-intersects is a non-empty set of Hausdorff dimension $n-1$, whereas the set where $\Sigma$ intersects itself tangentially is of Hausdorff dimension $n-2$. Thus there exists a point $q\in \partial c\cap E$ such that in an open neighborhood $B$ of $q$, $\Sigma$ is a finite union of embedded hypersurfaces $\tilde{H}_1$, $\tilde{H}_2$... with boundary in $\partial B$, intersecting two by two transversally. Consider a point $x\in B \cap E$ minimizing the number of hypersurfaces $\tilde{H}_j$ intersecting at $x$, among all the points $x'\in B\cap E$. We can assume that the hypersurfaces passing through $x$ are $\tilde{H}_1$,...,$\tilde{H}_L$. Let $B(x,r_x)\subset B$ be a small ball centered at $x$ such that $\partial B(x,r_x)$ intersects the $H_j$ transversally and $\Sigma\cap B(x,r_x) = \bigcup_{j=1}^L H_j$, where $H_j =\tilde{H}_j \cap \bar{B}(x,r_x)$ are closed $n$-disks intersecting two by two along $(n-1)$-disks. By the minimality property of $x$, for all $3\leq j \leq L$, $H_j \cap H_1 \subset H_1\cap H_2$ and so by symmetry
$$H_j \cap H_k = H_1\cap H_2 \qquad \forall 1\leq j,k \leq L.$$
Hence, the $H_j$ are $n$-disks intersecting along a same $(n-1)$-disk and subdividing $B(x,r_x)$ into $2L$ connected components, one of which at least is contained in $c$. This proves the local roof structure at $x$.

Let $\Omega$ be one component of $B(x,r_x)\backslash \bigcup_{j=1}^L H_j$ contained in $c$. By renumbering the $H_j$, we can assume that it is a connected component of $B(x,r_x)\backslash (H_1\cup H_2)$. In what follows, $j=1 \text{ or }2$, write $\iota_j:H_j\to M$ for the inclusion map, and $\nu_j$ for the unit normal of $H_j$ pointing toward $\Omega$. Write $\bar{\Omega} = \Omega \cup \partial \Omega$. We can apply Lemma \ref{justeacote} to each $H_j$ with a point $y_j$ replacing $p$, where $y_j\in H_j \backslash \bar{\Omega}$, and $U$ a neighborhood of $y_j$ disjoint from $\bar{\Omega}$. In this way, we get positive functions $f_j$ having small $C^1$ norm so that for $j\in\{1,2\}$:
\begin{itemize} 
\item $\tilde{\exp}_{\iota_j,f_j}(H_j) \subset B(x,r_x),$ 
\item in $\Omega$, each hypersurface $ \tilde{\exp}_{\iota_j,f_j}( H_j)$ has negative mean curvature, if endowed with the natural choice of outward unit normal given by $\nu_j$,
\item $\tilde{\exp}_{\iota_1,f_1}(H_1)$ and $\tilde{\exp}_{\iota_2,f_2}(H_2)$ meet transversally. 
\end{itemize}
Define for $s \in [0,1]$
$$F_s=\{x\in \tilde{\exp}_{\iota_j,s'f_j}(H_j) ; j\in\{1,2\}, s'\in[0,s]\} \cap \bar{\Omega},$$
$$G_s = c\backslash F_s.$$
We can use condition $\textbf{(LS)}$ to verify that we are indeed "pushing" $\partial c$ on one side. More precisely, we have for all $s \in (0,1]$:
$$\mathbf{M}(\partial [|G_s|]) = \mathcal{H}^n(\partial G_s) = \mathcal{H}^n(\partial c)+ \mathcal{H}^n( \partial F_s)-2\mathcal{H}^n(F_0).$$ 
The map $\xi : s \in [0,1] \mapsto  \partial [|G_s|]\in \mathcal{Z}_n(M,\mathbb{Z}_2)$ thus satisfies properties $(i)$ (by condition \textbf{(LS)}) and $(ii)$ of the lemma. Point $(iii)$ is a consequence of the first variation formula: since $H_j \cap \bar{\Omega} $ are minimal, only the boundary term appears and because the interior angle between these two hypersurfaces is less than $\pi$, we have
$$\frac{d}{ds}\bigg|_{s=0} \mathcal{H}^n(\partial G_s) <0,$$
so by taking $f_j$ even smaller if necessary, $(iii)$ is verified. Finally, point $(iv)$ will be checked in the Appendix (see Claim 2).

\end{proof}

Equipped with the previous two lemmas, we are now able to prove the following result: if $\Sigma$ is connected and non-embedded then for any $c\in \mathcal{C}_\Sigma$ satisfying \textbf{(LS)}, the boundary $\partial c$ can be approximated by a piecewise smooth mean convex hypersurface with less area. 

\begin{prop} \label{claim}
Let $\Sigma$ be a non-embedded connected minimal hypersurface in $M$. Consider a component $c\in \mathcal{C}_\Sigma$ satisfying condition \textbf{(LS)}. Then there exists a map $\theta : [0,1] \to \mathcal{Z}_n(M,\mathbb{Z}_2)$ continuous in the flat topology such that for all $s\in[0,1]$, $\theta(s) = \partial [|\rho_s|]$, where $\rho_s$ are open sets with the following properties: 
\begin{enumerate} [label=(\roman*)] 
\item $\rho_0 = c$, $\forall  s$ $\rho_s\subset c$ and  $\mathbf{M}(\theta(s))=\mathcal{H}^n(\partial \rho_s)$,
\item $\partial {\rho}_1$ is piecewise smooth mean convex,
\item $\forall s\in(0,1]$, $\mathcal{H}^n(\partial \rho_s) < \mathcal{H}^n(\partial c)$,
\item there exists a homotopically closed set of continuous sweepouts of ${\rho}_1$, called $\Lambda$,

\item $\theta$ satisfies $[*]$.
\end{enumerate}
\end{prop}

\begin{proof}

Suppose that $\Sigma= \psi(\Gamma)$ with $\Gamma$ connected, we will use the notations previously introduced. Let $\xi$ and $\{G_s\}_{s\in [0,1]}$ be the map and open sets constructed in Lemma \ref{Plateau}. By the local roof structure at a point of $\partial c$ (see first part of Lemma \ref{Plateau}), there is a point $z\in \partial c$ and a radius $r_0$ so that $B(z,r_0)\cap G_1=\varnothing$ and $\partial c$ is an embedded hypersurface in $B(z,r_0)$. Deform the metric of $M$ inside $B(z,r_0)$ and consider the positive function $f$ constructed in Lemma \ref{justeacote} with $z$ (resp. $B(z,r_0))$ replacing $p$ (resp. $U$). Then for $\lambda \in(0,1)$ small enough, $G_1 \cap \tilde{\exp}_{\tilde{\psi},\lambda {f}}(\tilde{\Gamma})$ is an immersed surface with positive mean curvature (for the original metric) with respect to the outward unit normal pointing toward the boundary of $c$. By \cite[Chapter 3, Proposition 3.2]{GG}, the set of generic immersions is dense in the space of smooth immersions, thus we can find a function $\tilde{f}$ arbitrarily close in the $C^\infty$ topology to $\lambda f$ such that $\tilde{\exp}_{\tilde{\psi},\tilde{f}}$ is a generic immersion and also generically meets $\partial G_1 \cap c$ (which is already generic by Lemma \ref{Plateau} (ii)). Let us choose $\lambda$ small and $\tilde{f}$ close enough to $\lambda f$ so that the mean convexity is preserved and the area does not increase by much. More precisely, if we define for all $s\inÊ[0,1]$
$$  V_s =  \{q\in \tilde{\exp}_{\tilde{\psi},s'\tilde{f}}(\tilde{\Gamma}) ; s'\in [0,s]\},$$
$$W_s = c \backslash V_s,$$
$$\rho_s = \left\{ \begin{array}{rcl}
G_{2s} & \mbox{for}
& s<1/2 \\ W_{2s-1} \cap G_1 & \mbox{for} & s\geq 1/2,
\end{array}\right.$$
then we can assure the following properties to be true:
\begin{enumerate} [label=(\alph*)]
\item $\partial\rho_1$ is piecewise smooth mean convex for the original metric,
\item $\forall s \in [1/2,1]$, $\mathcal{H}^n(\partial \rho_s) < \frac{1}{2}(\mathcal{H}^n(\partial G_1)+\mathcal{H}^n(\partial c)) < \mathcal{H}^n(\partial c).$
\end{enumerate} 
Note that item (b) can be satisfied using \textbf{(LS)} and Lemma \ref{Plateau} (iii): indeed outside a set of Hausdorff dimension $n-1$, $\partial c$ is an embedded hypersurface locally separating $c$ from $M\backslash c$ in the sense of \textbf{(LS)}, so $\partial c$ is only pushed on one side into $c$ by $\tilde{\exp}_{\tilde{\psi},s\tilde{f}}$, where $s \in [0,1]$. Points $(ii)$ and $(iii)$ of our proposition follow readily. Setting $\theta(s) = \partial[|\rho_s|]$, point $(i)$ is a consequence of condition \textbf{(LS)} again. We will check point $(iv)$, namely that there exists a homotopically closed set $\Lambda$ of continuous sweepouts of $\rho_1$, in the Appendix (see Claim 1). We will construct a continuous sweepout $\{S_t\}_{t\in I}$ of $\rho_1$ such that 
$$S_t = \{x\in \rho_1 \cup \partial \rho_1 ; u(x) = t\} \text{ for any } t\in I$$
where $u$ is a Morse function on $\rho_1$ continuous on $\rho_1 \cup \partial \rho_1$ with no critical points in a neighborhood of $\partial \rho_1$ and is obtained by "mollifying the distance function". Finally, the technical point $(v)$ will be also checked in the Appendix (see Claim 2).

\end{proof}

\begin{remarque}
It should be possible to prove that, with the notations in the proof of Proposition \ref{claim}, $s\mapsto \mathcal{H}^n(\partial W_s)$ is decreasing if $\tilde{f}$ is sufficiently small. But even if it assumed to be true, we need Lemma \ref{Plateau} to get rid of a part of $c$, deform the metric $g$ in this part and get a hypersurface being mean convex for the original metric $g$. 
\end{remarque}

The following lemma shows when and how one can retract the boundary of a $c\in \mathcal{C}_\Sigma$ to $0$ and will be crucial for constructing the optimal sweepout in the proof of Proposition \ref{embed}.

\begin{lemme} \label{attention}
Let $\Sigma$ be a non-embbedded connected minimal hypersurface in $M$ and suppose that $\mathcal{H}^n(\Sigma) \leq \mathfrak{A}_S(M) $. For any $c\in \mathcal{C}_\Sigma$, there is a map $\chi_c : I \to \mathcal{Z}_n(M,\mathbb{Z}_2)$ continuous in the flat topology such that:
\begin{enumerate} [label=(\roman*)]
\item $\chi_c(0)=\partial [|c|]$, $\chi_c(1)=0$ and $\spt \chi_c(s) \subset c$ for $s\in (0,1]$,
\item $\forall s\in(0,1]$, $\mathbf{M}(\chi_c(s)) < \mathbf{M}(\partial [|c|]) = \mathcal{H}^n(\partial {c})$,
\item for $j_0$ large enough and for any special chain map $\Phi : I(1,j_0) \to \mathbf{I}_*(M)$ determined by $\{\chi_c(s)\}_{s \in I}$ as in \cite{Alm1},  $$\sum\limits_{\alpha \in I(1,j_0) _1} \Phi (\alpha) = [|c|],$$
\item $\chi_c$ satisfies $[*]$.
\end{enumerate}

\end{lemme}

\begin{proof}

By Lemma \ref{partition}, $c$ satisfies condition \textbf{(LS)}. According to Proposition \ref{claim} $(iv)$, we can define $W( {\rho}_1, \partial  {\rho}_1,\Lambda)$. Suppose that we have 
\begin{equation} \label{inequ}
\mathcal{H}^n(\partial {\rho}_1) < W( {\rho}_1, \partial  {\rho}_1,\Lambda).
\end{equation}
By Theorem \ref{piecew}, there is an embedded connected minimal hypersurface $\Gamma_0 \subset \rho_1$. We have now a manifold $N$ whose boundary is the union of $\partial  {\rho}_1$ and a hypersurface isometric either to $\Gamma_0$ or to the double cover of $\Gamma_0$. The homology class of $\partial  {\rho}_1$ in $N$ is non-trivial and by \cite{Morgan} we can minimize the area to get an embedded stable minimal hypersurface $S$. Its area is not larger that $\mathcal{H}^n(\partial \rho_1)$, which is strictly smaller than $\mathcal{H}^n(\partial c)$. This is a contradiction with our assumption $\mathcal{H}^n(\Sigma) \leq \mathfrak{A}_S(M) $. So (\ref{inequ}) is false, i.e. in fact $\mathcal{H}^n(\partial {\rho}_1) = W( {\rho}_1, \partial  {\rho}_1,\Lambda)$ and one can find a family $\{T_t\}_{t\in [0,1]}\in \Lambda$ with 
$$\mathcal{H}^n(T_t)  < \mathcal{H}^n(\partial c)=  \mathbf{M}(\partial [|c|])\quad  \forall t.$$
Define $\chi_c(s)$ to be $\theta(2s)$ if $s\in [0,1/2]$ and the current in $\mathcal{Z}_n(M,\mathbb{Z}_2)$ determined by $T_{2s-1}$ if $s\in [1/2,1]$. 

Let's check that the continuous map $\chi_c:I\to \mathcal{Z}_n(M,\mathbb{Z}_2)$ satisfies the four conditions. Points $(i)$ and $(ii)$ in the conclusion of the theorem are clearly true by construction. Point $(iii)$ is also true. Indeed, $\{\chi_c(s)\}_{s\in [0,1]}$ "foliates" the open set $c$ and one can conclude by employing the methods in \cite[Theorem 5.8]{Zhou}. Finally the last technical condition $(iv)$ can be proved as follows. Firstly, Lemma \ref{claim} $(v)$ shows that 
$$\sup\{||\chi_c(x)||(B_r(p)) ; x \in [0,1/2], p\in M\} \to 0 \text{ as } r\to 0.$$
Then take a $\tau\in(0,1/2]$, Proposition 5.1 in \cite{Zhou} applies to $\chi_c\restriction_{[1/2+\tau,1]}$ and so 
$$\sup\{||\chi_c(x)||(B_r(p)) ; x \in [1/2+\tau,1], p\in M\} \to 0 \text{ as } r\to 0.$$
In the Appendix (see Claim 2), we will show that if $\tau$ is chosen small enough,
$$\sup\{||\chi_c(x)||(B_r(p)) ; x \in [1/2,1/2+\tau], p\in M\}\to 0 \text{ as } r\to 0,$$
which shows that indeed $\mathbf{m}(\chi_c,r)  \to 0$ as $r\to 0$. This ends the proof.

\end{proof}

We are now ready to show that a least area minimal hypersurface is necessarily embedded. The regularity result \cite{P} [Theorem $7.12$] will be used in the proof. Even though it is shown by Pitts only for $2\leq n\leq 5$, the result is still true for $n=6$ thanks to the curvature estimates of Schoen and Simon (see \cite[Section 7]{SchoenSimon}). Besides each time Theorem $7.12$ of \cite{P} and Theorem $5.5$ of \cite{Zhou} are invoked, we actually apply their "modulo 2" versions (see Remark \ref{modulo}).

\begin{prop} \label{embed}
Suppose that $\Sigma$ is a non-embbedded connected minimal hypersurface and that $\mathcal{H}^n(\Sigma) \leq \mathfrak{A}_S(M) $, then there exists a connected unstable minimal hypersurface $\Sigma_0$ such that:
\begin{enumerate} [label=(\roman*)]
\item $\Sigma_0$ is embedded,
\item $\mathcal{H}^n(\Sigma_0)=\mathbf{L}(\Pi)  < \mathcal{H}^n(\Sigma)$, where $\Pi$ is the homotopy class of mappings in $(\mathcal{Z}_n(M,\mathbf{M},\mathbb{Z}_2),\{0\})$ corresponding to the fundamental class in $H_{n+1}(M,\mathbb{Z}_2)$.

\end{enumerate}

\end{prop}

\begin{proof}

Let $\mathcal{C}_1$, $\mathcal{C}_2$ be the classes given by Lemma \ref{partition}. Define $A :I \to \mathcal{Z}_n(M,\mathbb{Z}_2)$ by
$$
A(s) = \left\{ \begin{array}{rcl}
 \sum\limits_{c \in \mathcal{C}_1} \chi_{c}(1-2s) & \mbox{for} & s \in [0,1/2] \\ 

  \sum\limits_{c \in \mathcal{C}_2} \chi_{c}(2s-1) & \mbox{for} & s \in [1/2,1], \\
\end{array}\right.$$
where $\chi_c$ are the maps constructed in Lemma \ref{attention}. It is well defined at $1/2$ because we are considering currents modulo $2$. By lemma \ref{attention}, this map $A: I \to \mathcal{Z}_n(M,\mathbb{Z}_2) $ is continuous in the flat topology and
$$\sup_{x \in I} \mathbf{M}(A(x)) <\infty \text{ and } \lim_{r \to 0} \mathbf{m}(A,r) = 0.$$

Thus $A$ determines by interpolation (\cite[Theorem 5.5]{Zhou}) a homotopy sequence of mappings $S= \{\phi_i\}_{i\in \mathbb{N}}$ where
$$\phi_i : I(1,n_i)_0 \to \mathcal{Z}_n(M,\mathbf{M},\mathbb{Z}_2) \text{ with fineness $\delta_i$, } \quad \lim\limits_{i\to \infty} \delta_i= 0,  \lim\limits_{i\to \infty} n_i = \infty.$$
We want to show that $S$ belongs to a homotopy class of mappings into $(\mathcal{Z}_n(M,\mathbf{M},\mathbb{Z}_2),\{0\})$, called $\Pi$, which is non-trivial. Because of \cite[Theorem 13.4]{Alm2} (or \cite[Theorem 4.6]{P}) and \cite[Theorem 8.2]{Alm1}, this will imply $\mathbf{L}(\Pi) >0$ so the min-max theory will produce non-trivial minimal hypersurfaces. For this purpose, recall that by \cite[Theorem 4.6]{P}, $$\pi_1^{\sharp}(\mathcal{Z}_n(M,\mathbf{M},\mathbb{Z}_2),\{ 0\}),  \pi_1^{\sharp}(\mathcal{Z}_n(M,\mathcal{F},\mathbb{Z}_2),\{ 0\}) \text{ and } \pi_1(\mathcal{Z}_n(M,\mathcal{F},\mathbb{Z}_2),\{ 0\})$$ are all naturally isomorphic. The map $A$ is continuous in the flat topology, so by restricting $A$ to $I(1,n_i)$ we obtain a $(1,\mathcal{F})$-homotopy sequence of mappings into $\pi_1^{\sharp}(\mathcal{Z}_n(M,\mathcal{F},\mathbb{Z}_2),\{0\})$, called $\tilde{A}$. Since by \cite[Theorem 5.5 (3)]{Zhou},
$$\sup\{\mathcal{F}(\phi_i(x) - A(x)) ; x \in I(1,n_i)\}\leq \delta_i,$$
$S$ determines the same $(1,\mathcal{F})$-homotopy class of mappings into $\mathcal{Z}_n(M,\mathcal{F},\mathbb{Z}_2),\{0\})$ as $\tilde{A}$, that is:
$$[S]=[\tilde{A}] \in\pi_1^{\sharp}(\mathcal{Z}_n(M,\mathcal{F},\mathbb{Z}_2),\{0\}).$$
$[\tilde{A}] $ is non-trivial if and only if the class $[A]$ is non-trivial in $\pi_1(\mathcal{Z}_n(M,\mathcal{F},\mathbb{Z}_2),\{ 0\}) $. But by Lemma \ref{attention} $(iii)$, the method described in \cite{Alm1} associates to $[A]$ the homology class 
$$\big{[}\hspace{0.1cm} -\sum\limits_{c \in {\mathcal{C}}_1} (-[|c|]) +  \sum\limits_{c \in {\mathcal{C}}_2} [|c|]  \hspace{0.1cm}\big{]} \in H_{n+1}(M,\mathbb{Z}_2)=\mathbb{Z}_2,$$
and this is equal to the non-zero fundamental class $[ \hspace{0.1cm}[|M|]\hspace{0.1cm} ]$, for $\mathcal{C}_1$ and $\mathcal{C}_2$ form a partition of $\mathcal{C}_\Sigma$. Consequently, $S$ belongs to the homotopy class of mappings $\Pi$ which satisfies $\mathbf{L}(\Pi) >0$.

By Lemma \ref{partition} and Lemma \ref{attention} $(ii)$, we have $\max\limits_{s \in I} \mathbf{M}(A(s)) \leq \mathcal{H}^n(\Sigma)$, so the interpolation theorem \cite[Theorem  5.5 (1)]{Zhou} implies that $\mathbf{L}(\Pi) \leq  \mathcal{H}^n(\Sigma)$. In the case where this inequality is strict then by \cite[Theorem 4.10, Theorem 7.12]{P}, there is an embedded minimal hypersurface $\Sigma_0$ (possibly disconnected and with multiplicity) whose area is $\mathbf{L}(\Pi)>0$. Since each connected component of $\Sigma_0$ is unstable, and by Proposition \ref{sweepout} has area larger than or equal to $\mathbf{L}(\Pi)$, it follows that $\Sigma_0$ is actually connected and $$\mathcal{H}^n(\Sigma_0) = \mathbf{L}(\Pi) <\mathcal{H}^n(\Sigma).$$

The case of equality is in fact impossible. Indeed, suppose $\mathbf{L}(\Pi) = \mathcal{H}^n(\Sigma)$, then $S = \{\phi_i\}$ is a critical sequence. Consider a sequence of slices $\phi_i(\alpha^i)$, where $  \alpha^i=[x_i] \in I(1,n_i)_0$, such that $$\mathbf{M}(\phi_i(\alpha^i)) \to \mathbf{L}(\Pi).$$
Because of \cite[Theorem 5.5 (1)]{Zhou}, necessarily $x_i\to 1/2$ so $\phi_i(\alpha^i)$ converge to $\Sigma$ in the flat topology by \cite[Theorem 5.5 (3)]{Zhou}. It is known that if $T_j \in \mathcal{Z}_k(M,\mathbb{Z}_2)$ converge to $T \in \mathcal{Z}_k(M,\mathbb{Z}_2)$ in the flat topology and the sequence of varifolds $|T_j|$ converge to $V \in \mathcal{V}_k(M)$, then $||V||(M) = \mathbf{M}(T)$ implies that $V=|T|$ (see \cite[Chapter 2, 2.1, (18), (f)]{P}). It follows that if $$\lim_i |\phi_i(\alpha^i)| = V \text{ and } ||V||(M) = \mathbf{L}(\Pi)= \mathcal{H}^n(\Sigma)$$ then $V=|\Sigma|$. Thus the only element of the critical set $\mathbf{C}(S)$ is $|\Sigma|$. But by \cite[Theorem $4.3$ (2), Theorem $4.10$]{P}, $|\Sigma|$ should be $\mathbb{Z}_2$ almost minimizing in small annuli around each point, and \cite[Theorem $7.12$]{P} again would imply that $\Sigma$ is embedded, a contradiction.

\end{proof}

We can now finish the proof of the main theorem.

\begin{proof} [Proof of Theorem \ref{principal} and Theorem \ref{principal bis}]
Let's first check the existence of a least area minimal hypersurface, using arguments appearing in \cite{MR}. By the compactness result in \cite{SchoenSimon}, $\mathfrak{A}_S(M)$ is achieved (take a minimizing sequence and apply the compactness theorem in balls of radius smaller than the injectivity radius of $M$). Thus if $$\mathfrak{A}_S(M)=\mathfrak{A}(M),$$ the existence of a minimizer is proved. Suppose on the contrary that $$\mathfrak{A}_S(M)>\mathfrak{A}(M),$$
and take a sequence $\{\Sigma_i\}_i$ of connected minimal hypersurfaces such that $\mathcal{H}^n(\Sigma_i)<\mathfrak{A}_S(M)$ and $\lim\limits_{i \to \infty} \mathcal{H}^n(\Sigma_i) = \mathfrak{A}(M)$. By Proposition \ref{sweepout} and Proposition \ref{embed}, each hypersurface $\Sigma_i$ has an area bigger than or equal to that of an embedded hypersurface, whose area is $\mathbf{L}(\Pi)$ where $\Pi$ is the homotopy class of mappings in $(\mathcal{Z}_n(M,\mathbf{M},\mathbb{Z}_2),\{0\})$ corresponding to the fundamental class in $H_{n+1}(M,\mathbb{Z}_2)$. Hence any minimal hypersurface produced by Almgren-Pitts' theory with the fundamental class of $H_{n+1}(M,\mathbb{Z}_2)$ has area $\mathfrak{A}(M)$ in the case $\mathfrak{A}_S(M)>\mathfrak{A}(M)$.

Then Proposition \ref{embed} implies that any least area minimal hypersurface is embedded. Moreover we can use Proposition \ref{sweepout} to show that if $\Sigma_0$ is a least area minimal hypersurface, it is either stable or coming from Almgren-Pitts' min-max theory with the fundamental class of $H_{n+1}(M,\mathbb{Z}_2)$. Finally, if $\Sigma_0$ is not stable, we have seen in the proof of Proposition \ref{sweepout} that $\Sigma_0$ is two-sided and by reasoning along the lines of \cite[Proposition 3.1]{MaNe} (see also \cite{Zhou}, \cite{MR}), $\Sigma_0$ is indeed of index one. This finishes the proof of Theorem \ref{principal bis}.

\end{proof}

\section{Area rigidity of minimal surfaces in three-manifolds of positive scalar curvature} \label{conjec}

As an application of Theorem \ref{principal bis}, we give a short proof of a conjecture of Marques and Neves (see Theorem 1.3 and below in \cite{MaNe}). Note that the proof only uses min-max methods and the short-time existence theorem for Hamilton's Ricci flow.  

\begin{theo}
Let $M^3$ be a closed three-manifold with scalar curvature $R$ at least $6$, not isometric to the round unit three-sphere $S^3$. Then there exists a closed embedded minimal surface $\Sigma$ of index zero or one such that
$$\mathcal{H}^2(\Sigma) < 4\pi.$$
Moreover, $\Sigma$ can be chosen so that if it is not stable, then it is two-sided and has area equal to the width of the fundamental class of $M$, in the sense of Almgren-Pitts.
\end{theo}

\begin{proof}
By \cite[Theorem 1.2]{MaNe} and Theorem \ref{principal bis}, the theorem is true for $M$ diffeomorphic to $S^3$. In the general case, it is enough to find a finite Riemannian covering $\tilde{M}$ of $M$ which contains a minimal surface $\tilde{\Sigma}$ of area less than $4\pi$. Indeed, if $p$ denotes the natural projection from $\tilde{M}$ to $M$, then $p(\tilde{\Sigma})$ is an immersed minimal surface of area less than $4\pi$ and the result follows readily from Theorem \ref{principal bis}. By \cite[Corollary 0.5, Chapter 15]{MorganTian} and \cite[Theorem 7.1, b)]{Marquesscalar}, an oriented cover $M_{or}$ of $M$ is a connected sum of spherical space forms and finitely many copies of $S^2\times S^1$. When this cover $M_{or}$ is a quotient of the three-sphere, we can just take $\tilde{M}=S^3$. Otherwise $M_{or}$ contains an essential two-sphere $S$. From Lemma 1 and Theorem 1 in \cite{MSY}, we can minimize the area of $S$ in its isotopy class by $\gamma$-reduction and get a non-trivial stable embedded minimal surface diffeomorphic to $S^2$ or $\mathbb{RP}^2$. By \cite[Proposition A.1, (i), (ii)]{MaNe}, it has area bounded by $4\pi/3$ or $2\pi$. Thus in this case we can take $\tilde{M}=M_{or}$.

\end{proof}

\section{Appendix} \label{claims}

In this section, we complete the proofs of Lemma \ref{Plateau}, Proposition \ref{claim} and Lemma \ref{attention} by proving two claims. Before stating the claims, recall the following convention. A subset $C$ of $\mathbb{R}^{n+1}$ is called graph of a real function $f$ over a domain $D$ of a hyperplane $H$ if there is an orthonormal basis $\{e_1,...,e_{n+1}\}$ of $\mathbb{R}^{n+1}$ satisfying:
\begin{itemize}
\item $\forall i\in \{1,...,n\}, e_i \in H$,
\item $f$ is defined on $\tilde{D}:=\{(x_1,...,x_n) ; \sum_{i=1}^n x_ie_i \in D\}$,
\item $C=\{\sum_{i=1}^n x_i e_i + f(x)e_{n+1} ; x=(x_1,...,x_n)\in \tilde{D}\}$.\\
\end{itemize}

\textbf{Claim 1:} Let $N$ be an open subset of $M$ such that $\partial N$ is piecewise smooth mean convex. Then there exists a homotopically closed family $\Lambda$ of sweepouts of $N$.\\

\begin{proof}[Proof of Claim 1]
Recall that by definition, a piecewise smooth mean convex hypersurface is two-sided, i.e. it locally separates $M$. 

It suffices to find one continuous sweepout, then we will take the homotopically closed family $\Lambda$ that it generates. This continuous sweepout will be determined by the level sets of a function $u$ which will be obtained by mollifying the distance function to $\partial N$, called $\mathbf{d}_{\partial N}$. Take $d_0>0$ smaller that the injectivity radius $\mathbf{inj}(M)$.

Consider a locally finite open cover of $\{p\in N : d(p, \partial N) \leq d_0)\}$. Note that this set 
does not intersect $\partial N$. We can suppose that the cover is given by $B(p_1,a_1)$, $B(p_2,a_2)$, ... where 
$$2a_i \leq \min\{d(p_i, \partial N), \mathbf{inj}(M)\} \quad \forall i\in \mathbb{N}\backslash\{0\}.$$
For any $i$, $\exp_{p_i}^{-1} : B(p_i,2a_i) \to \exp_{p_i}^{-1}(B(p_i,2 a_i))  \subset \mathbb{R}^{n+1}$ gives a coordinate chart. We want to modify $\mathbf{d}_{\partial N}$ via those charts. Consider a mollifier, namely a nonnegative smooth function $\vartheta : \mathbb{R}^{n+1} \to \mathbb{R}$ with support in the ball $B_{\mathbb{R}^{n+1}}(0,2)$, positive in $B_{\mathbb{R}^{n+1}}(0,1)$, such that $\int_{\mathbb{R}^{n+1}}\vartheta(\xi)d\xi =1$. To each ball $B(p_i,a_i)$ we associate:
\begin{itemize}
\item a small $b_i>0$ so that $\{x\in\mathbb{R}^{n+1}; B_{\mathbb{R}^{n+1}}(x,2 b_i) \subset \exp_{p_i}^{-1}(B(p_i,2 a_i)) \}$ contains $ \exp_{p_i}^{-1}(B(p_i,3a_i/2))$,
\item the mollifier $\vartheta_i (x) = 1/{b_i}^{n+1} \vartheta(x/b_i) $,
\item a smooth nonnegative cutoff function $\alpha_i$ with support in $B(p_i,3a_i/2)$, with values between $0$ and $1$ and equal to $1$ in $B(p_i,a_i)$.
\end{itemize}
We now construct by induction a sequence of functions $\tilde{d}^i$ which approximate $\mathbf{d}_{\partial N}$ and are smooth in respectively $\bigcup_{k=1}^i B(p_k,a_k)$. First we define $\tilde{{d}}^1 : N \cup \partial N \to \mathbb{R}$ with a slight abuse of notations:
$$
 \tilde{{d}}^1 = \alpha_1 \big\{(\exp_{p_1}^{-1})^*(\vartheta_1 \ast \exp_{p_1}^*(\mathbf{d}_{\partial N})) \big\} +(1-\alpha_1) \mathbf{d}_{\partial N}, 
 $$
where $\ast$ denotes the convolution product
$$(\exp_{p_1}^{-1})^* (\vartheta_1 \ast  \exp_{p_1}^*(\mathbf{d}_{\partial N}))(p) =\int_{\mathbb{R}^{n+1}}\vartheta_1( \exp_{p_1}^{-1}(p)-\xi)\mathbf{d}_{\partial N}(\exp_{p_1}(\xi)) d\xi.$$
Then if $\tilde{{d}}^i$ is constructed, we similarly define
\begin{equation} \label{convolutionary}
 \tilde{{d}}^{i+1} = \alpha_{i+1} \big\{(\exp_{p_{i+1}}^{-1})^*(\vartheta_{i+1} \ast \exp_{p_{i+1}}^*( \tilde{{d}}^i))\big\} +(1-\alpha_{i+1}) \tilde{{d}}^i.
\end{equation}
Since our cover is locally finite, $\tilde{d}^i$ locally simply converges to a function $\tilde{\mathbf{d}} : N \cup \partial N \to \mathbb{R}$, which is smooth and positive on the open set $U := \{p\in N : d(p, \partial N) < d_0)\}$, equal to $0$ on $\partial N$ and continuous on $U \cup \partial N$.

At any point $p$ of $\partial N$, $\partial N$ is locally contained in the union of a finite number of embedded hypersurfaces intersecting at $p$. Hence we could have taken $d_0$ small enough so that for any point $p \in U \cup \partial N$, $\mathbf{d}_{\partial N}$ is locally the minimum of the distance to a finite number of embedded hypersurfaces $\Omega_1$, ..., $\Omega_J$, that is for $p'$ near $p$:
\begin{equation} \label{dist}
\mathbf{d}_{\partial N}(p') = \min_{i=1,...,J} d(p',\Omega_i).
\end{equation}
We can reduce $d_0$ again if necessary, so that each function $d(.,\Omega_i)$ is smooth in a neighborhood of $p$. Moreover because $\partial N$ is piecewise smooth mean convex, we can find a constant $\kappa_1>0$ such that if $p \in \partial N$, then there is $\mathbf{v}_p\in T_pM$ such that $\langle \mathbf{v}_p, \nu \rangle_g\geq \kappa_1$, for any $\nu\in T_pM$ outward unit normal at $p$ of one of the smooth pieces $\Omega_i$. Another useful remark is that $\mathbf{d}_{\partial N}$ is differentiable almost everywhere and if it is differentiable at $p\in U$, its differential is equal to the differential of one of the $d(.,\Omega_i)$ ($\Omega_i$ being as in (\ref{dist})). Note that by (\ref{convolutionary}),
\begin{align*}
\nabla \tilde{d}^{i+1} & = \alpha_{i+1} \nabla \big\{(\exp_{p_{i+1}}^{-1})^*(\vartheta_{i+1} \ast \exp_{p_{i+1}}^*( \tilde{{d}}^i))\big\} + (1- \alpha_{i+1}) \nabla \tilde{d}^i\\
& + \nabla\alpha_{i+1} (\big\{(\exp_{p_{i+1}}^{-1})^*(\vartheta_{i+1} \ast \exp_{p_{i+1}}^*( \tilde{{d}}^i))\big\} - \tilde{d}^i).
\end{align*}
By the usual properties of convolution with a Lipschitz function and the equality above, if we choose $b_i$ successively small enough, then the limit gradient $\nabla \tilde{\mathbf{d}}$ will be arbitrarily close to a local average (in a $d(p,\partial N)$-neighborhood of $p$) of the differentials of $d(.,\Omega_i)$. Summing up all these facts, one obtains for $b_1$, $b_2$... small enough:
\begin{enumerate} \label{wxc}
\item $\nabla \tilde{\mathbf{d}}$ does not vanish in $U$,
\item $|| \nabla \tilde{\mathbf{d}}||$ is bounded by $2$ in $U$,
\item at a point $p\in \partial N$, there is a $v(p)\in \mathbb{R}^{n+1}$ and a small radius $r(p)<\mathbf{inj}(M)$ such that for $p'\in U\cap B(p,r(p))$, 
$$\langle d_{p'}\exp_p^{-1}(\nabla \tilde{\mathbf{d}}(p')), v(p) \rangle_{st}> \kappa_1/2 $$
where $\langle.,.\rangle_{st}$ denotes the standard scalar product in $\mathbb{R}^{n+1}$, and $d_{p'}\exp_p^{-1}$ is the differential of $\exp_p^{-1}$ at $p'$.
\end{enumerate}
Items (1)-(3) imply that there exist $K>0$, $r_0>0$ verifying the following property:\\ 
\begin{align} \label{lipschitz}
\begin{split} 
&\text{For all $p \in \partial N$ and $p' \in B(p,r_0)\cap U$,}\\
&\exp_p^{-1} (B(p,r_0)\cap \tilde{\mathbf{d}}^{-1}(\tilde{\mathbf{d}}(p'))) \subset \mathbb{R}^{n+1}\\
&\text{is the graph of a $K$-Lipschitz smooth function defined}\\
&\text{over a domain of a hyperplane independent of $p'$.}\\ 
&\text{The function and its domain depend smoothly on $p'$}.
\end{split}
\end{align}

It remains to construct $u$ such that it coincides with $\tilde{\mathbf{d}}$ near $\partial N$ and is a Morse function in $N$. It is similar to the proof of Claim 1 in \cite{Zhou}. Define
$$V_s = \{p\in U\cup \partial N ; \tilde{\mathbf{d}}(p)<s\}.$$
For $\epsilon>0$ small enough, $V_{2\epsilon} \subset U\cup \partial N$ and there exists a smooth function $h$ defined on $N$ arbitrarily close to $\tilde{\mathbf{d}}$ in $V_{2\epsilon}\backslash V_{\epsilon}$ for the $C^1$ topology, and such that $h(a)>\tilde{\mathbf{d}}(b)$ for any $a\in N\backslash V_{\epsilon}$ and $b\in V_{\epsilon/2}$. The function $h$ can be assumed to be Morse because of the density of Morse functions in $C^k(N)$ for $k\geq2$. Consider a cutoff function $\varphi : N\to \mathbb{R}$ such that $\varphi \equiv 1$ on $V_\epsilon$,  $\varphi \equiv 0$ on $N\backslash V_{2\epsilon}$ and $0\leq \varphi\leq 1$. Define
$$u = \varphi \tilde{\mathbf{d}} +(1-\varphi) h.$$
If $h$ was chosen sufficiently close to $\tilde{\mathbf{d}}$ in $V_{2\epsilon}\backslash V_{\epsilon}$ for the $C^1$ topology, then item (1) in the previous list implies that $\nabla u$ does not vanish in $V_{2\epsilon}$. Reparametrizing the level sets of $u$, we get a family $\{\Gamma_t\}_{t\in I}$.

Note that for any smooth point $p$ of $\partial N$, there is a neighborhood $B(p,r')$ of $p$ inside which ${\mathbf{d}}_{\partial N}$ is smooth with non vanishing gradient and so the level sets of $\tilde{\mathbf{d}}$ become closer and closer (in the $C^\infty$-topology) to the level sets of $\mathbf{d}_{\partial N}$ near $\partial N$. Consequently, $B(p,r')\cap \tilde{\mathbf{d}}^{-1}(\tilde{\mathbf{d}}(p'))$ converges smoothly to $B(p,r')\cap \partial N$ as $p' \to p$. So adding Property (\ref{lipschitz}), we readily check that $\{\Gamma_t\}_{t\in I}$ satisfies the conditions of Definition \ref{definition} and is the wanted sweepout.

\end{proof}

Consider a family of closed subsets $\{C_s\}_{s\in I}$ of $M$. Suppose that there are a finite family of points $\{p_j\}_{j=1}^J$ and a corresponding family of radii $\{r_j\}_{j=1}^J$ such that $r_j < \mathbf{inj}(M)$, the balls $B(p_j,r_j)$ cover $M$ and there is a positive integer $K_0$ such that for all $j\in \{1,...,J\}$ and $s\in I$,
$$\exp_{p_j}^{-1}(B(p_j,r_j) \cap C_s) \subset \mathbb{R}^{n+1}$$
is included in the union of at most $K_0$ graphs of $K_0$-Lipschitz functions defined over domains of possibly different hyperplanes. In this case, we will say that the family $\{C_s\}$ is uniformly Lipschitz. Then we have the following easy claim:\\

\textbf{Claim 2:} If $\Phi:[a,b] \to \mathcal{Z}_n(M,\mathbb{Z}_2)$ is continous in the flat topology and if $\{\spt(\Phi(x))\}_{x\in [a,b]}$ is a uniformly Lipschitz family of closed sets then 
$$\mathbf{m}(\Phi,r) \to 0 \text{ as } r\to 0.$$

\begin{proof}[Proof of Claim 2]
It is enough to verify it in each ball $B(p_j,r_j)$. But then the lemma follows from the formula for computing the area in $\mathbb{R}^{n+1}$ of the graph of a real function defined on a domain of $\mathbb{R}^n$ and the fact that there is a uniform constant $\kappa_2>0$ with $\exp_{p_j}^*g \leq \kappa_2 . g_{st}$ for all $j$ ($g$ is the metric on $M$ and $g_{st}$ is the standard metric on $\mathbb{R}^{n+1}$).
\end{proof}

The previous claim then suffices to complete the proofs of Lemma \ref{Plateau}, Proposition \ref{claim} and Lemma \ref{attention}. Indeed, in Lemma \ref{Plateau} and Proposition \ref{claim}, the supports of $\xi$ and $\theta$ clearly form a uniformly Lipschitz family. Secondly, for Lemma \ref{attention}, we just use Property (\ref{lipschitz}).  

\bibliographystyle{plain}
\bibliography{biblio06_12_16memoire}

\end{document}